\documentclass[final,12pt,a4paper]{elsarticle}
\usepackage{float}
\usepackage{tikz}
\DeclareRobustCommand\full  {\tikz[baseline=-0.6ex]\draw[thick] (0,0)--(0.5,0);}
\DeclareRobustCommand\dotted{\tikz[baseline=-0.6ex]\draw[thick,dotted] (0,0)--(0.54,0);}

\usepackage{tikz}
\newlength\replength
\newcommand\repfrac{.40}% PERCENT OF REPETITION USED BY DASH
\newcommand\rulewidth{.6pt}% DASH WIDTH
\def\dashht{.5\dimexpr\ht\strutbox-\dp\strutbox\relax}
\newcommand\tdashfill[1][\repfrac]{\cleaders\hbox to \replength{%
  \smash{\rule[\dashht]{\repfrac\replength}{\rulewidth}%
  \kern.5\dimexpr\replength-\repfrac\replength-2.5pt\relax%
  \raisebox{\dimexpr\dashht-.3pt}{.}}}\hfill}

\usepackage{eurosym}
\usepackage{slashbox}
\usepackage{graphics}
\newcommand{\nn}{\nonumber}

\newcommand{\D}{\Delta}

\usepackage{bbm}
\usepackage{subcaption}
\usepackage{amssymb}
\usepackage{mathrsfs}
\usepackage{mathtools}
\usepackage{amsmath}
\usepackage{xcolor}
\definecolor{darkgreen}{rgb}{0.0, 0.2, 0.13}

\newcommand{\norma}[1]{{\left\|#1\right\|}}

\usepackage{cases}
\usepackage{amsfonts}
\usepackage{amssymb}
\usepackage{amsthm}
\usepackage{graphicx}
\usepackage{mathrsfs}
\usepackage{xcolor}
\usepackage{latexsym}
\usepackage{tikz}
\usepackage{lineno}
\usepackage{caption}
\usepackage[colorlinks,plainpages=true,pdfpagelabels,hypertexnames=true,colorlinks=true,pdfstartview=FitV,linkcolor=black,citecolor=black,urlcolor=black]{hyperref}
\PassOptionsToPackage{unicode}{hyperref}
\PassOptionsToPackage{naturalnames}{hyperref}
\usepackage{enumerate}
\usepackage[shortlabels]{enumitem}
\usepackage{bookmark}
\usepackage{esint}
\usepackage[ddmmyyyy]{datetime}
\usepackage[margin=2cm]{geometry}
\makeatletter
\g@addto@macro\normalsize{%
  \setlength\abovedisplayskip{4pt}
  \setlength\belowdisplayskip{4pt}
  \setlength\abovedisplayshortskip{4pt}
  \setlength\belowdisplayshortskip{4pt}
}
\numberwithin{equation}{section}
\everymath{\displaystyle}
\usepackage[capitalize,nameinlink]{cleveref}
\crefname{section}{Section}{Sections}
\crefname{subsection}{Subsection}{Subsections}
\crefname{condition}{Condition}{Conditions}
\crefname{hypothesis}{Hypothesis}{Conditions}
\crefname{assumption}{Assumption}{Assumptions}
\crefname{lemma}{Lemma}{Lemmas}
\crefname{definition}{Definition}{Definitions}

\crefformat{equation}{\textup{#2(#1)#3}}
\crefrangeformat{equation}{\textup{#3(#1)#4--#5(#2)#6}}
\crefmultiformat{equation}{\textup{#2(#1)#3}}{ and \textup{#2(#1)#3}}
{, \textup{#2(#1)#3}}{, and \textup{#2(#1)#3}}
\crefrangemultiformat{equation}{\textup{#3(#1)#4--#5(#2)#6}}%
{ and \textup{#3(#1)#4--#5(#2)#6}}{, \textup{#3(#1)#4--#5(#2)#6}}%
{, and \textup{#3(#1)#4--#5(#2)#6}}

\Crefformat{equation}{#2Equation~\textup{(#1)}#3}
\Crefrangeformat{equation}{Equations~\textup{#3(#1)#4--#5(#2)#6}}
\Crefmultiformat{equation}{Equations~\textup{#2(#1)#3}}{ and \textup{#2(#1)#3}}
{, \textup{#2(#1)#3}}{, and \textup{#2(#1)#3}}
\Crefrangemultiformat{equation}{Equations~\textup{#3(#1)#4--#5(#2)#6}}%
{ and \textup{#3(#1)#4--#5(#2)#6}}{, \textup{#3(#1)#4--#5(#2)#6}}%
{, and \textup{#3(#1)#4--#5(#2)#6}}

\crefdefaultlabelformat{#2\textup{#1}#3}
\newtheorem{theorem}{Theorem}[section]
 \newtheorem{lemma}[theorem]{Lemma}
\newtheorem{corollary}[theorem]{Corollary}

 \newtheorem{definition}[theorem]{Definition}% Use {\rm ...}
\numberwithin{equation}{section}

\def\N{\mathbb{N}}
\def\CC{{\rm \kern.24em \vrule width.02em height1.4ex depth-.05ex \kern-.26emC}}

\def\TagOnRight
\def\AA{{it I} \hskip-3pt{\tt A}}

\def\QQ{\rlap {\raise 0.4ex \hbox{$\scriptscriptstyle |$}} {\hskip -0.1em Q}}

\makeatletter
\newcommand{\vo}{\vec{o}\@ifnextchar{^}{\,}{}}
\makeatother
\def\YYint#1#2#3{{\setbox0=\hbox{$#1{#2#3}{\iint}$}
    \vcenter{\hbox{$#2#3$}}\kern-.50\wd0}}
\def\XXint#1#2#3{{\setbox0=\hbox{$#1{#2#3}{\int}$}
    \vcenter{\hbox{$#2#3$}}\kern-.50\wd0}}

\makeatletter
\def\namedlabel#1#2{\begingroup
   \def\@currentlabel{#2}%
   \label{#1}\endgroup
}
\makeatother
\makeatletter
\newcommand{\rmh}[1]{\mathpalette{\raisem@th{#1}}}
\newcommand{\raisem@th}[3]{\hspace*{-1pt}\raisebox{#1}{$#2#3$}}
\makeatother

\newcommand{\descref}[2]{\hyperref[#1]{\textnormal{\textcolor{black}{(}\textcolor{black}{\bf #2}\textcolor{black}{)}}}}
\newcommand{\dref}[2]{\hyperref[#1]{\textcolor{black}{(}\textcolor{black}{\bf #2}\textcolor{black}{)}}}
\newcommand{\be} {\begin{eqnarray}}
\newcommand{\ee} {\end{eqnarray}}
\newcommand{\Bea} {\begin{eqnarray*}}
\newcommand{\Eea} {\end{eqnarray*}}

\newcommand{\dott}{\, \cdot\,}

\newcommand{\R}{\mathbb{R}}

\newcommand{\sgn}{\mathop\mathrm{sgn}}
\renewcommand{\L}[1]{\mathbf{L^#1}}

\newcommand{\Z}{\mathbb{Z}}

\DeclareMathOperator{\lip}{Lip}
\DeclareMathOperator{\TV}{TV}
\newcommand{\abs}[1]{\left| #1\right|}

\newcounter{whitney}
\refstepcounter{whitney}

\newcounter{ineqcounter}
\refstepcounter{ineqcounter}
\makeatletter
\def\ps@pprintTitle{%
\let\@oddhead\@empty
\let\@evenhead\@empty
\def\@oddfoot{}%
\let\@evenfoot\@oddfoot}
\makeatother
\usepackage[titletoc,toc,page]{appendix}

\usepackage{pgf,tikz}

\allowdisplaybreaks[4]

\renewcommand{\epsilon}{\varepsilon}
\renewcommand{\L}[1]{\mathbf{L^#1}}

\newcommand{\modulo}[1]{{\left|#1\right|}}

\newcommand{\interi}{{\mathbb{Z}}}

\renewcommand{\d}[1]{\mathinner{\mathrm{d}{#1}}}

\usetikzlibrary{arrows}
\usetikzlibrary{decorations.pathreplacing}
\begin{document}

\begin{frontmatter}

% \title{Accuracy of the finite volume
% approximations to nonlocal conservation laws}
\title{On the accuracy of the finite volume approximations to nonlocal conservation laws}
%\title{Rate of the finite volume approximations to nonlocal conservation laws}

\author[myaddress1]{Aekta Aggarwal}
\ead{aektaaggarwal@iimidr.ac.in}

\address[myaddress1]{Operations Management and Quantitative Techniques, Indian Institute of Management\\Prabandh Shikhar, Rau--Pithampur Road, Indore, Madhya Pradesh 453556, India}

\author[myaddress2]{Helge Holden}
\ead{helge.holden@ntnu.no}

 \address[myaddress2]
 {Department of Mathematical Sciences, 
NTNU --- Norwegian University of Science and Technology,\\ NO–7491 Trondheim, Norway.}

\author[myaddress2]{Ganesh Vaidya}
\ead{ganesh.k.vaidya@ntnu.no}

\begin{abstract}
In this article, we discuss the error analysis for  a certain class of monotone finite volume schemes approximating nonlocal scalar conservation laws, modeling traffic flow and crowd dynamics, without any additional assumptions on monotonicity or linearity of the kernel $\mu$ or the flux $f$. We first prove a novel Kuznetsov-type lemma for this class of PDEs and thereby show that the finite volume approximations converge to the entropy solution at the rate of $\sqrt{\Delta t}$ in $L^1(\R)$. To the best of our knowledge, this
is the first proof of any type of convergence rate for this class of conservation laws.
 We
also present numerical experiments to illustrate this result. \end{abstract}


\end{frontmatter}
\section{Introduction}
% \fbox{Please change numbering so that Thm, Lemma, Corr, Def etc are numbered together. No separate number for each type!}
The celebrated
Lighthill--Whitham--Richards (LWR) model{\color{blue}~\cite{lighthill1955kinematic,richards1956shock}} given by
\begin{equation}
 \label{LWR}
u_t+(u \nu(u))_x=0,\,\qquad (t,x) \in (0,\infty)\times \R,   
\end{equation}
with $u$ being the mean traffic density and $\nu$ the mean traffic speed, is one of the widely used models in traffic flow modeling. However, being a non-linear hyperbolic conservation law, it can have solutions with discontinuities and infinite accelerations, adversely impacting its capability to capture  physical traffic phenomena effectively. Thus, over last decade, the parallel class of conservation laws with \textit{nonlocal} parts in the flux is gaining particular interest in the modeling as well as mathematical community, where a convolution is introduced in the flux to produce Lipschitz-continuous velocities ensuring
bounded accelerations. The two most popular strategies are to evaluate the traffic speed, from either averaging the traffic density, or
 averaging over the velocity, leading to the following two conservation laws, 
 \begin{align}
u_t+\Big(u\nu(u*\mu)\Big)_x&=0,\label{134}\\
u_t+\Big(u(\nu(u)*\mu)\Big)_x&=0,\label{135}
% u_t+(u(\eta*\mathcal{V}(u)))_x&=&0
\end{align} 
where $\nu,\mu \in (C^2 \cap W^{2,\infty}) (\R),$ and 
\[(\mu *\; u) (t,x)=
  \displaystyle\displaystyle \int\limits_{\R}\mu (x-\xi) \; u (t,\xi) \, \d\xi.\]
Such conservation laws have been studied in the recent literature, see \cite{friedrich2018godunov,ColomboGaravelloLecureux2012, ColomboLecureuxPerDafermos,aggarwal2016crowd,aggarwal2016crowd,blandin2016well,keimer2018nonlocal,keimer2017existence,friedrich2022conservation,coclite2023nonlocal}, and the references therein. From the point of view of modeling, this \emph{nonlocal} nature is particularly suitable in describing the behavior of
traffic, where each vehicle moves according to its evaluation of the density and its variations within its horizon. 

 The article studies a very general class of these nonlocal conservation laws,  namely
\begin{align}
 \label{IVP:eq}
  \partial_t u
  +
  \partial_x (f(u) \nu(\mu*\beta(u)))&= 0,\,\quad\quad \quad \quad \quad \,\,\,(t,x) \in Q_T:=(0,T)\times \R,
  \\ \label{IVP:data}
  u(0,x)&=u_0(x), \quad\quad \quad \,\, \,\,\,x \in \R,
\end{align}  
where 
\begin{enumerate}[label=(\textbf{H\arabic*})]
	\item \label{A1} $f\in \lip(\R)$  with $ f(0)=0,$ 
 \item $\beta, \nu \in (C^2 \cap   \operatorname{BV}  \cap \, W^{2,\infty}) (\R),$ with $\nu(0)=\beta(0)=0,$
 \item $\mu \in (C^2 \cap   \operatorname{BV}  \cap \, W^{2,\infty}) (\R),$
 \end{enumerate}
with $f$ being non-linear (in contrast to \eqref{134} or \eqref{135}).
 They serve as working models for a variety of real life applications, for example, sedimentation
models~\cite{betancourt2011nonlocal}, crowd dynamics
models~\cite{ColomboGaravelloLecureux2012, ColomboHertyMercier,
ColomboLecureuxPerDafermos}, vehicular
traffic \cite{blandin2016well,
ColomboHertyMercier}, biological
applications in structured population dynamics~\cite{Perthame2007}, supply chain models\cite{ColomboHertyMercier}, granular material dynamics \cite{AmadoriShen2012}, as well as  conveyor belt dynamics ~\cite{gottlich2014modeling}.

The wellposedness of this class has been of interest in the last few years. The local counterpart of \eqref{IVP:eq}--\eqref{IVP:data} enjoys a rich literature, with \cite{kruvzkov1970first} as one of the pioneering papers to fix the wellposedness of the entropy solutions for such PDEs. 
% \begin{definition}
% A function $u\in C([0,\infty);L^1(\R))\cap L^{\infty}(\overline{Q}_T)$  is a weak solution of IVP~\eqref{IVP:eq}--\eqref{IVP:data} if \begin{eqnarray*}\label{weak1}\begin{aligned}
% \displaystyle \int_Q \left(u(t,\:x) \phi_t(t,\:x)+f(u)\mathcal{U}(t,x)\phi_x(t,\:x)\right) \d{t}\d{x}+\int\limits_{\R} u_0(x)\phi(0,x) \d x =0\:\forall \phi \in C_c^{\infty}(\overline{Q}_T) \end{aligned},\end{eqnarray*} with $\mathcal{U}(t,x)= \nu(\mu*\beta(u(t))(x))$. 
% \end{definition}
Similar to its local counterpart, since $f$ can be possibly nonlinear, there can be multiple weak solutions of IVP~\eqref{IVP:eq}--\eqref{IVP:data}.   Hence, an additional entropy condition is required to single out a unique solution.
\begin{definition}\label{ent}
    A function $u\in C([0,T];L^1(\R))\cap L^{\infty}(\overline{Q}_T)$  is an entropy solution of the IVP~\eqref{IVP:eq}--\eqref{IVP:data}, 
if for every $k\in\R,$ and for all non-negative $\phi\in C_c^{\infty}([0,T)\times \R)$,
\begin{align}\nonumber
&\int\limits_{Q_T} |u(t,x)-k|\phi_t\d t \d x +\int\limits_{Q_T}\sgn (u(t,x)-k) \mathcal{U}(t,x)(f(u)-f(k))\phi_x \d t \d x \\ 
&-\int\limits_{Q_T} f(k) \sgn (u(t,x)-k)\mathcal{U}_x(t,x)\phi \d t \d x
+\int\limits_{\R} |u_0(x)-k|\phi(0,x) \d x \geq 0,\label{kruz}
\end{align}
where $\mathcal{U}(t,x)=\nu(\mu*\beta(u(t))(x))$.

\end{definition}
With some appropriate modifications, the proof of \cite{betancourt2011nonlocal} can be adapted to prove that any two entropy solutions satisfying Definition \ref{ent} are equal, while existence of these solutions has been proven in \cite{AmorimColomboTeixeira, aggarwal2015nonlocal} for non linear $f$ and in \cite{boudin2012numerical} for linear $f$, via the convergence of finite volume approximations. These articles dealing with existence of solutions establish that the schemes converge to the entropy solution $u \in C
  \left([0,T];L^1(\R)\right)$. 
  
What remains unexplored is  to analyze the rate of convergence, i.e., how fast the error $||u^{\Delta}(T,\dott)-u(T,\dott)||_{L^1(\R)}$ made by the numerical solution $u^{\Delta}$ in approximating the exact solution
$u$  goes to zero as the mesh size $\Delta x$ goes to zero.  That is the precise aim of this article. In other words, we look for an (optimal) $\alpha$ satisfying
\begin{equation}\label{alpha}
||u^{\Delta}(T,\dott)-u(T,\dott)||_{L^1(\R)} \leq C {\Delta x}^{\alpha},
\end{equation}
with $C$ being an appropriate positive constant. 
% \textbf{****We follow the approach of Kuznetsov\cite{Kuznetsov} and*******}
To achieve this, we first prove a Kuznetsov-type lemma using the entropy formulation \eqref{kruz}. We further estimate the relative entropy functional involving the solution $u$ and numerical approximation $u^{\Delta}$ to obtain \eqref{alpha} with an optimal $\alpha=1/2,$ same as the one obtained in \cite{sabac_optimal,Kuznetsov} for local fluxes (homogeneous).  To the best of our knowledge, this is the first result in this direction for such nonlocal conservation laws. It is to be noted that the results of the article hold under no additional assumptions on monotonicity/linearity of the kernel $\mu$ or the flux $f$ or $\nu$.

The paper is organized as follows. In Section \ref{WP}, we discuss the wellposedness of \eqref{IVP:eq}--\eqref{IVP:data} via convergence of a general class of monotone finite volume approximations. In Section \ref{EE}, we prove the Kuznetsov-type lemma for \eqref{IVP:eq}--\eqref{IVP:data} and obtain the rate of convergence as $1/2.$ In Section \ref{exis2}, we also briefly comment on the extensions to higher dimensions. In Section \ref{NR}, we present some numerical experiments which illustrate the theory. 
\section{Finite volume approximations and wellposedness}\label{WP} 
% {\color{blue} In this section, we review the existence and uniqueness of solutions for \eqref{IVP:eq}--\eqref{IVP:data}, studied in \cite{aggarwal2015nonlocal,betancourt2011nonlocal} and references therein.}

We now introduce the notations to be used in the article:
\begin{enumerate}[i.]
\item $|u|_{L^\infty_t \operatorname{BV}_x}:=\sup_{t\in[0,T]} \TV(u(t,\dott)).$ \item $|u|_{\lip_t L^1_x}:=\sup_{0\leq t_1<t_2\leq T}\frac{\norma{u(t_1,\dott)-u(t_2,\dott)}_{L^1(\R)}}{|t_1-t_2|}.$ 
    \item $K:=\{u:\overline{Q}_T \rightarrow \R: ||u||_{L^{\infty}(\overline{Q}_T)}+|u|_{L^\infty_t \operatorname{BV}_x}<\infty\}.$
    % \item$\gamma(u,\sigma)=\sup_{\substack{
    % 0\leq t_1\leq t_2\leq T \\ \abs{t_1-t_2} \leq \sigma}} \norma{u(t_1)-u(t_2)}_{L^(\R)}.$ 
  \item$\gamma(u,\sigma):=\sup_{\substack{
    \abs{t_1-t_2} \leq \sigma\\  0\leq t_1< t_2 \leq T }} \norma{u(t_1,\dott)-u(t_2,\dott)}_{L^1(\R)}.$
\end{enumerate}
\subsection{Uniqueness of the entropy solution}
Any two entropy solutions of the IVP \eqref{IVP:eq}--\eqref{IVP:data} are equal. More precisely, we have the following result:
\begin{theorem}[Uniqueness]
Let $u,v\in C([0,T];L^1(\R)) \cap (L^\infty_t \operatorname{BV}_x) (\overline{Q}_T)$ be two entropy solutions of the IVP \eqref{IVP:eq}--\eqref{IVP:data} corresponding to the initial data ${u_0}$ and ${v_0}$ respectively.
Then, there exists a constant 
$$
\mathcal{M}=\mathcal{M}(f,\mu, \eta, \nu, \beta, \norma{u}_{L^1(Q_T)},\norma{v}_{L^1(Q_T)},|u|_{L^\infty_t \operatorname{BV}_x},|v|_{L^\infty_t \operatorname{BV}_x},|u|_{L^\infty_t \operatorname{BV}_x},T)>0
$$
such that
\begin{equation*}
\norma{u(t,\dott)-v(t,\dott)}_{L^1(\R)} \leq \norma{u_0-v_0}_{L^1(\R)}(1+\mathcal{M}t\exp(\mathcal{M}t)), \quad t\in[0,T].
\end{equation*}
In particular, if $u_0=v_0,$ then $u=v$ a.e.~in $\overline{Q}_T$.
\end{theorem}
\begin{proof}
Note that the nonlocal coefficient of the flux function considered in this article is more general than the ones in all the previous results on the uniqueness of nonlinear-nonlocal conservation laws. However, the continuous dependence estimates for the entropy solution of the conservation laws (local)  derived in \cite[Thm.~1.3]{karlsen2003uniqueness} can still be invoked to prove the desired weighted contraction estimates as in
    \cite[Thm.\ 4.1]{betancourt2011nonlocal}. Alternatively, this can also be seen as a consequence of the Kuznetsov-type estimates derived in the sequel~(see Lemma~\ref{lemma:kuz}), by sending $\epsilon,\epsilon_0 \rightarrow 0.$
    \end{proof}
\subsection{Existence via numerical approximations}\label{exis}
For $\Delta x,\Delta t>0,$ and $\lambda:=\Delta t/\Delta x,$ consider equidistant spatial grid points $x_i:=i\Delta x$ for $i\in \Z$  and temporal grid points $t^n:=n\Delta t$ 
for non-negative integers $n\le N$, such that $T =N\Delta t$.  Let $\chi_i(x)$ denote the indicator function of $C_i:=[x_{i-1/2}, x_{i+1/2})$, where $x_{i+1/2}=\frac{1}{2}(x_i+x_{i+1})$ and let
$\chi^n(t)$ denote the indicator function of $C^{n}:=[t^n,t^{n+1})$. We approximate the initial data according to:
\begin{equation}\label{apx:data}
u^{\Delta}_0(x):=\sum_{i\in\Z}\chi_i(x)u^{0}_i\quad \mbox{where }u^{0}_i=\int\limits_{C_i}u_0(x)\d x\mbox{ for }i\in\Z.
\end{equation}

We define a piecewise constant approximate solution $u^{\Delta}$ to~\eqref{IVP:eq}
by
\[
  u^{\Delta} (t,x) =  u^{n}_{i}
  \quad \mbox{ for } \quad
(t,x)\in C^{n}\times C_i, \quad n\in\N, i\in\Z, 
\]
through the following marching formula:
\begin{align}\nn
   u^{n+1}_i&=H(\nu(c^{n}_{i-1/2}),\nu(c^{n}_{i+1/2}),u_{i-1}^{n},u_i^{n},u_{i+1}^{n})
    \\ \nn 
    & := 
   u^{n}_i- \lambda \bigl[
    \mathcal{F} (\nu(c^n_{i+1/2}),u_i^{n},u_{i+1}^{n})
    - 
    \mathcal{F} (\nu(c^n_{i-1/2}),u_{i-1}^{n},u_{i}^{n})
     \bigr]\\ 
     \label{scheme2}
     &:=u^{n}_i- \lambda \bigl[
    \mathcal{F}^n_{i+1/2} (u_i^{n},u_{i+1}^{n})
    - 
    \mathcal{F}^n_{i-1/2} (u_{i-1}^{n},u_{i}^{n})
     \bigr], 
\end{align} 
where
the convolution term $\mu*\beta(u)$ is computed through a standard quadrature formula using the same space mesh, i.e.,
\begin{equation}\label{eq:conv}
c_{\strut i+1/2}^{n}:=\Delta x\sum\limits_{p\in Z} \mu_{i+1/2-p}\beta(u^{n}_{p+1/2}) \approx \int\limits_{\R} \mu(x_{i+1}-y)\beta(u^{\D}(t^n)(y))\d y, 
\end{equation}
with $u^{n}_{p+1/2}$ being any convex combination of
$u^{n}_{p}$ and $u^{n}_{p+1},$  and 
 $\mu_{i+1/2} = \mu (x_{i+1/2}).$
Further, $\mathcal{F} (\nu(c^n_{i+1/2}),u_i^{n},u_{i+1}^{n})
    $ denotes the numerical approximation of the flux $f(u) \nu(\mu*\beta(u))$ at the interface $x=x_{i+1/2}$ for $i\in \Z,$ with $H$ being increasing in the last three arguments. The  approximations generated by the scheme, namely $u_i^{n} \approx u(t^n,x_i)$
are extended to a function defined on ${Q}_T$ via
\begin{equation}\label{def_u_De}
u^{\Delta}(t,x) =\sum_{n=0}^{N-1} \sum_{i\in \mathbb{Z}} \chi_i(x) \chi^n(t) u_i^{n}.
\end{equation}
In general, $\mathcal{F}$ can be defined as an appropriate nonlocal extension of any monotone numerical flux, meant for local conservation laws. Here, we present examples of celebrated Lax--Friedrichs flux and Godunov flux.
\begin{enumerate}
    \item \textbf{Lax--Friedrichs type flux:} For any $ \theta \in \left(0,\frac{2}{3} \right),$ define
\begin{equation*}
  \mathcal{F}_{\rm LF}(a,b,c)
   = 
  \frac{a}{2}\Big( f(b)
    +
    f(c)\Big)
  -
  \theta\frac{(c-b)}{2\, \lambda},
\end{equation*} 
% \fbox{how to choose $\theta$?} 
where  $\Delta t$ is chosen in order to satisfy
the CFL condition
\begin{equation}\label{CFL_LF}
   \lambda \le \displaystyle\frac{\min(1, 4-6\theta,6\theta )}{1+6\abs{f}_{\lip(\R)}\norma{\nu}_{L^\infty(\R)}}.\end{equation}
\item \textbf{Godunov type flux:}
\begin{equation*}
\mathcal{F}_{\rm Godunov}(a,b,c)= a F_{\rm Godunov}(b,c),
\end{equation*} 
where the function $F_{\rm Godunov}$ is the Godunov flux for the corresponding local conservation law $u_t+f(u)_x=0$, and  $\Delta t$ is chosen in order to satisfy
the CFL condition
\begin{equation*}
   \lambda \abs{f}_{\lip(\R)}\norma{\nu}_{L^\infty(\R)}\leq \frac16.
   \end{equation*}
% \end{eqnarray*}

\end{enumerate}
\begin{theorem}[Existence] \label{thm:2.2} Assume that (\textbf{H1})--(\textbf{H3}) hold.
For non-negative initial data $u_0\in (L^1\cap BV) (\R)$ and $\D x >0,$ there exist constants $\mathcal{L}_i, i=1,2,3,$ independent of $\D t$ such that the sequence of approximations $u^n$ defined by \eqref{scheme2} satisfies the following for all $i\in \Z, n\in \N$ and $0\leq n \leq N$:
\begin{enumerate}
\item Positivity:
\begin{equation}\label{apx:positivity}
u_i^n\geq0.
\end{equation}
\item $L^{\infty}$ estimate:
\begin{equation}\label{apx:Linf}
\norma{u^n}_{L^{\infty}}\leq \exp(\mathcal{L}_1 T)\norma{u^0}_{L^{\infty}}.
\end{equation}
\item $L^{1}$ estimate:
\begin{equation}\label{apx:L1}
\norma{u^{n}}_{L^1}=\norma{u^0}_{L^1}.
\end{equation}
\item BV estimate:
\begin{equation}\label{apx:bv}
\TV(u^{n}) \leq \exp(\mathcal{L}_2T)(\TV(u^{0}) +\mathcal{L}_2).
\end{equation}
\item Time continuity:
\begin{equation}\label{apx:time}
\D x \sum\limits_{i \in \Z}\abs{u_i^m-u_i^n} \leq \mathcal{L}_3|m-n| \D t, \quad \quad m,n \in \N\cup\{0\}.
\end{equation}
\item Discrete entropy inequality:
For any $k\in\R$ we have
\begin{align} \nn
\modulo{u_i^{n+1}-k}-\modulo{u_i^{n}- k}+&\lambda\big(\mathcal{G}^{n}_{i+1/2}(u_i^{n} ,u_{i+1}^{n},k )-\mathcal{G}^{n}_{i-1/2}(u_{i-1}^{n} ,u_i^{n},k)\big)\\
 &\qquad+\lambda\sgn(u_i^{n+1}-k) f(k)(\nu(c_{i+\frac{1}{2}}^{n})-\nu(c_{i-\frac{1}{2}}^{n}))\le 0, \label{apx:de}
 \end{align}
 where $\mathcal{G}^{n}_{i+1/2}(a,b,k)=\mathcal{F}_{i+1/2}^n(a \vee k,b \vee k)-\mathcal{F}_{i+1/2}^n(a \wedge k,b \wedge k)$ for all $i\in\Z, n\in\N$.
 \end{enumerate}
Furthermore, the finite volume approximations converge to the unique entropy solution $u$ of the IVP \eqref{IVP:eq}--\eqref{IVP:data}.
\end{theorem}
\begin{proof}
The proof follows by invoking the monotonicity of the scheme and writing it in the incremental form.  The details can be worked out exactly on the similar lines of~\cite[Lem.\ 2.4--2.7]{aggarwal2015nonlocal} and~\cite[Lem.\ 2.2--2.8]{AmorimColomboTeixeira} with proper modification in the estimations on the nonlocal coefficient. 
% \fbox{More details}
\end{proof}
The above theorem implies that the entropy solution satisfies the following regularity estimates. 
\begin{corollary}[Regularity of the entropy solution]Assume that (\textbf{H1})--(\textbf{H3}) hold.
For $0 <t \leq T$ and $u_0\in (L^1\cap BV) (\R),$ the entropy solution $u$ of the IVP \eqref{IVP:eq}--\eqref{IVP:data} satisfies the following:
\begin{align*}
\norma{u(t,\dott)}_{L^{\infty}(\R)}&\leq \exp(\mathcal{L}_1T)\norma{u_0}_{L^{\infty}(\R)}\\
\norma{u(t,\dott)}_{L^1(\R)}&=\norma{u_0}_{L^1(\R)}\\
\TV(u(t,\dott)) &\leq \exp(\mathcal{L}_2T)(\TV(u_0) +\mathcal{L}_2)\\
\norma{u(t_2,\dott)-u(t_1,\dott)}_{L^1(\R)} &\leq \mathcal{L}_3\abs{t_2-t_1}, \text{ where } 0\leq t_1,t_2 \leq T.
\end{align*}
\end{corollary}
 
% ***My be we can introduce the notations like $|u|_{BV(Q_T)}$ and $|u|{\lip}$****
\section{Error estimate}\label{EE}

 We define $\Phi: \overline{Q}_T^2 \rightarrow \R$ by  
 \begin{equation*}
\Phi(t,x,s,y):=\Phi^{\epsilon,\epsilon_0}(t,x,s,y)=\omega_{\epsilon}(x-y)\omega_{{\epsilon}_0}(t-s),
\end{equation*}
where $\omega_a(x)=\frac{1}{a}\omega\big(\frac{x}{a}\big),$ $a>0$ and $\omega$ is a standard symmetric mollifier with $\operatorname{supp} (\omega) \subseteq [-1,1].$ Furthermore, we assume that $\int_\R \omega_a(x) \d x =1$ and $\int_\R\abs{\omega'_a(x)} \d x =\frac{1}{a}.$ Now, it is straight forward to see that $\Phi$ is symmetric and
$\Phi_x=\omega'_{{\epsilon}}(x-y)\omega_{\epsilon_0}(t-s)=-\Phi_y,  \Phi_t=\omega_{\epsilon}(x-y)\omega'_{{\epsilon}_0}(t-s)=-\Phi_s$.
% \end{enumerate}
Further, define the following functions.
\begin{definition}
 \begin{align*}
 G(a,b)&:=\sgn (a-b) (f(a)-f(b)),\\
% \label{1}
    \Lambda_T(u,\phi,k)&:= \int_{Q_T}\Big( |u-k|\phi_{t}+\mathcal{U}(t,x)G(u,k)\phi_{x}-  \sgn (u-k) f(k)\mathcal{U}_x(t,x)\phi\Big) \d t \d x\\
    &\quad -\int_{\R}|u(T,x)-k|\phi(T,x)\d x+\int_{\R}|u_0(x)-k|\phi(0,x)\d x, \quad \text{where }\phi\in C_c^{\infty}(\overline{Q}_T),\\ 
    \Lambda_{\epsilon,\epsilon_0}(u,v)&:=\int_{Q_T}\Lambda_T(u,\Phi(\dott,\dott,s,y), v(s,y))\d y \d s.
    % ,\\
    %   \Lambda_{\epsilon,\epsilon_0}(v,u)&:=\int_{Q_T}\Lambda_T(v,\Phi(t,x,\dott,\dott), u(t,x))\d t \d x.
  \end{align*}
  % \fbox{Are $\Lambda_{\epsilon,\epsilon_0}(u,v)$ and $\Lambda_{\epsilon,\epsilon_0}(v,u)$ the same?} 
  \end{definition}

We now state and prove the Kuznetsov-type lemma for nonlocal conservation laws.
\begin{lemma}\label{lemma:kuz}
Let $u$ be the entropy solution of \eqref{IVP:eq}--\eqref{IVP:data} and $v \in K.$ Then,
\begin{align*}
\norma{u(T,\dott)-v(T,\dott)}_{L^1(\R)}&\le\mathcal{K}\left(-\Lambda_{\epsilon,\epsilon_0}(v,u)+ \norma{u_0-v_0}_{L^1(\R)}+\gamma(v,{\epsilon_0})+\epsilon+\epsilon_0 \right), 
\end{align*}
  where $\mathcal{K}$ is constant that depends on
  %\begin{align*}\mathcal{K}=\mathcal{K}(|f|_{\lip(\R)},\norma{\mu}_{L^{\infty}(\R)},|\eta|_{\lip(\R)},|\nu|_{\lip(\R)},|\beta|_{\lip(\R)},\norma{u}_{L^1(Q_T)},\norma{v}_{L^1(Q_T)},|u|_{L^\infty_t \operatorname{BV}_x},|v|_{L^\infty_t \operatorname{BV}_x},|u|_{L^\infty_t \operatorname{BV}_x},T)\end{align*} 
  \begin{align*}
  \mathcal{K}=\mathcal{K}(f,\mu, \eta, \nu, \beta, \norma{u}_{L^1(Q_T)},\norma{v}_{L^1(Q_T)},|u|_{L^\infty_t \operatorname{BV}_x},|v|_{L^\infty_t \operatorname{BV}_x},|u|_{L^\infty_t \operatorname{BV}_x},T)
  \end{align*}
and is independent of $\epsilon,\epsilon_0$.
\end{lemma}
 
 \begin{proof}
Consider the sum $\Lambda_{\epsilon,\epsilon_0}(u,v)+\Lambda_{\epsilon,\epsilon_0}(v,u)$:
\begin{align*}
\Lambda_{\epsilon,\epsilon_0}(u,v)&+\Lambda_{\epsilon,\epsilon_0}(v,u) \\
&\quad=\int_{Q^2_T}\Big( |u(t,x)-v(s,y)|\Phi_{t} +\mathcal{U}(t,x)G(u,v)\Phi_{x} \Big) \d t \d x \d y \d s \\&\qquad- \int_{Q^2_T} f(v) \sgn (u(t,x)-v(s,y)) \mathcal{U}_x(t,x)\Phi \d t \d x \d y \d s\\
    &\qquad-\int_{Q_T}\int_{\R}|u(T,x)-v(s,y)|\Phi(T,x,s,y)\d x \d y \d s \\&\qquad+\int_{Q_T}\int_{\R}|u_0(x)-v(s,y)|\Phi(0,x,s,y)\d x \d y \d s\\
   &\qquad+ \int_{Q^2_T}\Big( |u(t,x)-v(s,y)|\Phi_{s} +\mathcal{V}(s,y)G(u,v)\Phi_{y}\Big) \d t \d x \d y \d s \\ &\qquad - \int_{Q^2_T} f(u) \sgn (v(s,y)-u(t,x)) \mathcal{V}_y(s,y)\Phi \d t \d x \d y \d s\\
    &\qquad-\int_{Q_T}\int_{\R}|v(T,y)-u(t,x)|\Phi(t,x,T,y)\d x \d y \d s \\&\qquad+\int_{Q_T}\int_{\R}|v_0(y)-u(t,x)|\Phi(t,x,0,y)\d x \d y \d s,
 \end{align*}    
where $\mathcal{V}(s,y)= \nu(\mu*\beta(v(s))(y)).$ 
Since $\Phi_{s}=-\Phi_{t},\Phi_{y}=-\Phi_{x}, $ we have 
\begin{align*}
\Lambda_{\epsilon,\epsilon_0}&(u,v)+\displaystyle\Lambda_{\epsilon,\epsilon_0}(v,u)\\
&\quad=\int_{Q^2_T}G(u,v)\Phi_{x}(\mathcal{U}(t,x)-\mathcal{V}(s,y)) \d t \d x \d y \d s\\
&\qquad-\int_{Q^2_T}\sgn (u(t,x)-v(s,y)) \Big(f(v)  \mathcal{U}_x(t,x)-f(u)  \mathcal{V}_y(s,y)\Big)\Phi\d t \d x \d y \d s\\
    &\qquad-\int_{Q_T}\int_{\R}|u(T,x)-v(s,y)|\Phi(T,x,s,y)\d x \d y \d s\\&\qquad+\int_{Q_T}\int_{\R}|u_0(x)-v(s,y)|\Phi(0,x,s,y)\d x \d y \d s\\ 
    &\qquad-\int_{Q_T}\int_{\R}|v(T,y)-u(t,x)|\Phi(t,x,T,y)\d y \d t \d x\\ &\qquad+\int_{Q_T}\int_{\R}|v_0(y)-u(t,x)|\Phi(t,x,0,y)\d y \d t \d x.
% \\&& \text{Using the symmetry of }\Phi\\
% &=&{\color{black}\int_{Q^2_T}G(u,v)\Phi_{x}(\mathcal{U}(t,x)-\mathcal{V}(s,y)) \d t \d x \d y \d s}\\&&-\int_{Q^2_T}\sgn (u(t,x)-v(s,y)) \Big(f(v)  \mathcal{U}_x(t,x)-f(u)  \mathcal{V}_y(s,y)\Big)\Phi\d t \d x \d y \d s\\
% &&-{\color{black}\int_{Q_T}\int_{\R}|u(T,x)-v(s,y)|\Phi(s,x,T,y)\d y \d x \d s+\int_{Q_T}\int_{\R}|u_0(x)-v(s,y)|\Phi(s,x,0,y)\d x \d y \d s}\\
%     &&-{\color{black}\int_{Q_T}\int_{\R}|v(T,y)-u(t,x)|\Phi(s,x,T,y)\d y \d t \d x+\int_{Q_T}\int_{\R}|v_0(y)-u(t,x)|\Phi(s,x,0,y)\d y \d t \d x}.
    \end{align*}
 In other words, we have
 \begin{equation*}
 \Lambda_{\epsilon,\epsilon_0}(u,v)=
 -\Lambda_{\epsilon,\epsilon_0}(v,u)+I_{\Phi'}+I_{\Phi}+I_0-I_T,
 \end{equation*}  
 with 
 \begin{align*}
I_{\Phi'}&=\int_{Q^2_T}G(u,v)\Phi_{x}(\mathcal{U}(t,x)-\mathcal{V}(s,y)) \d t \d x \d y \d s,\\[2mm]
I_{\Phi}&=-\int_{Q^2_T}\sgn (u(t,x)-v(s,y)) \Big(f(v)  \mathcal{U}_x(t,x)-f(u)  \mathcal{V}_y(s,y)\Big)\Phi\d x dt \d y \d s,\\
I_T&=\int_{Q_T}\int_{\R}\Big(|u(T,x)-v(t,y)|+|v(T,y)-u(t,x)|\Big)\Phi(t,x,T,y)\d y \d t \d x,\\[2mm]
I_0&=\int_{Q_T}\int_{\R}\Big(|u_0(x)-v(t,y)|+|v_0(y)-u(t,x)|\Big)\Phi(t,x,0,y)\d y \d x \d t.
 \end{align*}
  Since $u$ is the entropy solution of \eqref{IVP:eq}--\eqref{IVP:data}, we have that $\displaystyle\Lambda_{\epsilon,\epsilon_0}(u,v)\ge 0, $ and hence,
  \begin{align}\label{kuz}
I_T&\le-\displaystyle\Lambda_{\epsilon,\epsilon_0}(v,u) + I_{\Phi'}+I_{\Phi}+I_0.
 \end{align}
The terms $I_0$ and $I_T$ appear in the local case as well so they can be estimated on the similar lines of \cite{holden2015front,GTV_NM} to get:
\begin{align}\label{1c}
   I_T&\ge \norma{u(T,\dott)-v(T,\dott)}_{L^1(\R)}-\mathcal{K}_1(\epsilon+\epsilon_0+\gamma(v,\epsilon_0)),\\
   \label{1d}
   I_0&\le \norma{u_0-v_0}_{L^1(\R)}+\mathcal{K}_1(\epsilon+\epsilon_0+\gamma(v,\epsilon_0)),
\end{align}
where $\mathcal{K}_1=\mathcal{K}_1(\abs{u}_{L^\infty_t \operatorname{BV}_x},\abs{v}_{L^\infty_t \operatorname{BV}_x}, |u|_{L^\infty_t \operatorname{BV}_x} )$.
 Now, we estimate the other terms one by one.
 Using integration by parts $I_{\Phi'}$ can be written as,
 \begin{equation*}
I_{\Phi'}=-\int_{Q^2_T}\Phi\big[G_x(u,v)(\mathcal{U}(t,x)-\mathcal{V}(s,y))+G(u,v)\mathcal{U}_x(t,x) \big]\d t \d x \d y \d s.
% \\[2mm]
% &=&-{\color{black}\int_{Q^2_T}\Phi(G_x(u,v)(\mathcal{U}(t,x)-\mathcal{V}(s,y))+\sgn(u-v)(f(u)-f(v))\mathcal{U}_x(t,x) )\d t \d x \d y \d s}.\\
\end{equation*}
Consequently,
\begin{align*}
I_{\Phi'}&+I_{\Phi}\\&=-\int_{Q^2_T}\Phi(G_x(u,v)(\mathcal{U}(t,x)-\mathcal{V}(s,y))+\sgn(u-v)(f(u)-f(v))\mathcal{U}_x(t,x) )\d t \d x \d y \d s\\[2mm]
&\quad-\int_{Q^2_T}\sgn (u(t,x)-v(s,y)) \Big(f(v)  \mathcal{U}_x(t,x)-f(u)  \mathcal{V}_y(s,y)\Big)\Phi\d x dt \d y \d s\\[2mm]
  &=-\int_{Q^2_T}\Phi(G_x(u,v)\big[(\mathcal{U}(t,x)-\mathcal{V}(s,y))+\sgn(u-v)f(u)(\mathcal{U}_x(t,x)-\mathcal{V}_y(s,y))\big]\d t \d x \d y \d s\\
  &:= I_{\mathcal{U}}+I_{\mathcal{U}_x}. 
\end{align*}
Consider the term
\begin{equation*}
I_{\mathcal{U}}=\int_{Q^2_T}\Phi\big[G_x(u,v)(\mathcal{V}(s,y)-\mathcal{U}(t,x))\big]\d{x}\d{t}  \d{y}\d{s}.
\end{equation*}
Since $|G_x(u,v)|\le \abs{f}_{\lip(\R)}\abs{u_x}$ (in the sense of measures, see~\cite[Lem.\ A2.1]{BK98} for details), we have, 
\begin{equation*}
 I_{\mathcal{U}}\le \abs{f}_{\lip(\R)}\int_{Q^2_T}\Phi\abs{u_x}\abs{\mathcal{V}(s,y)-\mathcal{U}(t,x)}\d{x} \d{t}\d{y} \d{s}.
\end{equation*}
Note that the term 
\begin{align*}
&|\mathcal{V}(s,y)-\mathcal{U}(t,x)|  \\
&\le\nn\abs{\mathcal{V}(s,y)-\mathcal{V}(s,x)}+\abs{\mathcal{V}(s,x)-\mathcal{U}(t,x)}\\
&=|\nu(\mu*\beta(v(s)))(y)-\nu(\mu*\beta(v(s)))(x)|+|\nu(\mu*\beta(v(s)))(x)-\nu(\mu*\beta(u(t)))(x)|\\
&\le\abs{\nu}_{\lip(\R)}\abs{(\mu*\beta(v(s)))(y)-(\mu*\beta(v(s)))(x)}\\&\qquad+\abs{\nu}_{\lip(\R)}\abs{(\mu*\beta(v(s)))(x)-(\mu*\beta(u(t)))(x)}\\
&=\abs{\nu}_{\lip(\R)}\Big(\abs{\int_{\R}\beta(v(s,z))(\mu(y-z)-\mu(x-z))\d z}\\
&\qquad+
\abs{\int_{\R}(\beta(v(s,z))-\beta(u(t,z)))\mu(x-z)\d z}\Big)\\
&=\abs{\nu}_{\lip(\R)}\abs{\beta}_{\lip(\R)}\abs{\mu}_{\lip(\R)}\norma{v(s,\dott)}_{L^1(\R)}\abs{y-x}\\&\qquad+\abs{\nu}_{\lip(\R)}\abs{\beta}_{\lip(\R)} \norma{\mu}_{L^{\infty}(\R)}
\norma{v(s,\dott)-u(t,\dott)}_{L^1(\R)}. 
\end{align*}
Consequently we get:
\begin{align*}
I_{\mathcal{U}}&\leq\abs{f}_{\lip(\R)}\abs{\nu}_{\lip(\R)}\abs{\beta}_{\lip(\R)} \int_{Q^2_T}\Phi\abs{u_x}\Big(\abs{\mu}_{\lip(\R)}\norma{v(s,\dott)}_{L^1(\R)}\abs{y-x}\\
&\quad+
\norma{\mu}_{L^{\infty}(\R)}\norma{v(s,\dott)-u(t,\dott)}_{L^1(\R)}\Big)\d{x} \d{y}\d{t} \d{s}%\label{IU} 
\\
&\leq\abs{f}_{\lip(\R)}\abs{\nu}_{\lip(\R)}\abs{\beta}_{\lip(\R)}\abs{\mu}_{\lip(\R)}\int_{Q^2_T}\Phi\abs{u_x}\norma{v(s,\dott)}_{L^1(\R)}\abs{y-x}\d{x} \d{y}\d{t} \d{s}\\
&\quad+\abs{f}_{\lip(\R)}\abs{\nu}_{\lip(\R)}\abs{\beta}_{\lip(\R)}\norma{\mu}_{L^{\infty}(\R)}\int_{Q^2_T}\Phi\abs{u_x}
\norma{u(s,\dott)-u(t,\dott)}_{L^1(\R)}\d{x} \d{y}\d{t} \d{s}\\
&\quad+\abs{f}_{\lip(\R)}\abs{\nu}_{\lip(\R)}\abs{\beta}_{\lip(\R)} \norma{\mu}_{L^{\infty}(\R)}\int_{Q^2_T}\Phi\abs{u_x}
\norma{v(s,\dott)-u(s,\dott)}_{L^1(\R)}\d{x} \d{y}\d{t} \d{s}\\
&:=I^1_{\mathcal{U}}+I^2_{\mathcal{U}}+I^3_{\mathcal{U}},
\end{align*}
where $I^1_{\mathcal{U}},I^2_{\mathcal{U}}$, and $I^3_{\mathcal{U}}$ satisfy the following estimates.
 \begin{align*}
I^1_{\mathcal{U}}&=\abs{f}_{\lip(\R)}\abs{\nu}_{\lip(\R)}\abs{\beta}_{\lip(\R)}\abs{\mu}_{\lip(\R)}\\&\qquad \times \int_{Q^2_T}\omega_{\epsilon}(x-y)\omega_{{\epsilon}_0}(t-s)\abs{u_x}\norma{v(s,\dott)}_{L^1(\R)}\abs{y-x}\d{x}
\d{y}\d{t}\d{s}\\
&\leq\abs{f}_{\lip(\R)}\abs{\nu}_{\lip(\R)}\abs{\beta}_{\lip(\R)}\abs{\mu}_{\lip(\R)}\int_{Q^2_T}\omega_{\epsilon}(x-y)\omega_{{\epsilon}_0}(t-s)\abs{u_x}\norma{v(s,\dott)}_{L^1(\R)}\epsilon\d{x}
\d{y}\d{t}\d{s}\\
&\leq\abs{f}_{\lip(\R)}\abs{\nu}_{\lip(\R)}\abs{\beta}_{\lip(\R)}\abs{\mu}_{\lip(\R)}\abs{u}_{L^\infty_t \operatorname{BV}_x}\norma{v}_{L^1(Q_T)}\epsilon, 
\\[3mm]
I^2_{\mathcal{U}}&=\abs{f}_{\lip(\R)}\abs{\nu}_{\lip(\R)}\abs{\beta}_{\lip(\R)}\norma{\mu}_{L^{\infty}(\R)}\\&\qquad \times \int_{Q^2_T}\omega_{\epsilon}(x-y)\omega_{{\epsilon}_0}(t-s)\abs{u_x}
\norma{u(s,\dott)-u(t,\dott)}_{L^1(\R)}\d{x} \d{y}\d{t}\d{s}\\
&\le \abs{f}_{\lip(\R)}\abs{\nu}_{\lip(\R)}\abs{\beta}_{\lip(\R)} \abs{u}_{L^\infty_t \operatorname{BV}_x}\norma{\mu}_{L^{\infty}(\R)}\\&\qquad\times\int_{Q^2_T}\omega_{\epsilon}(x-y)\omega_{{\epsilon}_0}(t-s)\abs{u_x}
|t-s|\d{x} \d{y}\d{t}\d{s}\\
&\leq \abs{f}_{\lip(\R)}\abs{\nu}_{\lip(\R)}\abs{\beta}_{\lip(\R)}\abs{u}_{L^\infty_t \operatorname{BV}_x} \abs{u}_{L^\infty_t \operatorname{BV}_x}\norma{\mu}_{L^{\infty}(\R)} \int_0^T\int_0^T\omega_{{\epsilon}_0}(t-s)
{\epsilon_0}\d{t}\d{s}\\
&=\abs{f}_{\lip(\R)}\abs{\nu}_{\lip(\R)}\abs{\beta}_{\lip(\R)}\norma{\mu}_{L^{\infty}(\R)}\abs{u}_{L^\infty_t \operatorname{BV}_x} \abs{u}_{L^\infty_t \operatorname{BV}_x} T
{\epsilon_0},
\\[3mm]
I^3_{\mathcal{U}}
&=\abs{f}_{\lip(\R)}\abs{\nu}_{\lip(\R)}\abs{\beta}_{\lip(\R)}\abs{u}_{L^\infty_t \operatorname{BV}_x}\norma{\mu}_{L^{\infty}(\R)}\int_{0}^T\int_{0}^T\omega_{{\epsilon}_0}(t-s)
 \norma{v(s,\dott)-u(s,\dott)}_{L^1(\R)}\d{t} \d{s}
\\
&\leq\abs{f}_{\lip(\R)}\abs{\nu}_{\lip(\R)}\abs{\beta}_{\lip(\R)}\abs{u}_{L^\infty_t \operatorname{BV}_x}\norma{\mu}_{L^{\infty}(\R)}\int_{0}^T
\norma{v(s,\dott)-u(s,\dott)}_{L^1(\R)}\d{s}.
\end{align*}
Collectively,
we have
\begin{align}
 I_{\mathcal{U}}&\le \abs{f}_{\lip(\R)}\abs{\nu}_{\lip(\R)}\abs{\beta}_{\lip(\R)}\abs{u}_{L^\infty_t \operatorname{BV}_x}\nonumber\\
 &\quad \times\left(\abs{\mu}_{\lip(\R)}\norma{v}_{L^1(Q_T)}\epsilon+\norma{\mu}_{L^{\infty}(\R)}(\abs{u}_{L^\infty_t \operatorname{BV}_x}T
\epsilon_0+\int_{0}^T
\norma{v(s,\dott)-u(s,\dott)}_{L^1(\R)}\d{s})\right)\nonumber\\
&\le  \mathcal{K}_2(\epsilon+\epsilon_0)+\mathcal{K}_3\int_{0}^T
\norma{v(s,\dott)-u(s,\dott)}_{L^1(\R)}\d{s},\label{IUf}
\end{align}
for some appropriate constants $\mathcal{K}_2$ and $\mathcal{K}_3$.
Now, we consider,
\begin{align*}
 I_{\mathcal{U}_x}&=\int_{Q^2_T}\Phi\sgn(u-v)f(u)(\mathcal{U}_x(t,x)-\mathcal{V}_y(s,y))\d t \d x \d y \d s\\
 &\le \abs{f}_{\lip(\R)}\int_{Q^2_T}\omega_{\epsilon}(x-y)\omega_{{\epsilon}_0}(t-s)\abs{u}\abs{\mathcal{U}_x(t,x)-\mathcal{V}_y(s,y)}\d t \d x \d y \d s.
 \end{align*}
 Note that 
\begin{equation*}
|\mathcal{V}_y(s,y)-\mathcal{U}_x(t,x)|\le\nn\abs{\mathcal{V}_y(s,y)-\mathcal{V}_x(s,x)}+\abs{\mathcal{V}_x(s,x)-\mathcal{U}_x(t,x)}.
% &=&\underbrace{\abs{\nu'((\mu*\beta(v))(s,y))(\mu'*\beta(v))(s,y)-\nu'((\mu*\beta(v))(s,x))(\mu'*\beta(v))(s,x)}}_{\abs{\mathcal{V}_y(s,y)-\mathcal{V}_x(s,x)}}\\
% &&+\underbrace{\abs{\nu'((\mu*\beta(v))(s,x))(\mu'*\beta(v))(s,x)-\nu'((\mu*\beta(u))(t,x))(\mu'*\beta(u))(t,x)}}_{\abs{\mathcal{V}_x(s,x)-\mathcal{U}_x(t,x)}}
\end{equation*}
Now, adding and subtracting $\nu'((\mu*\beta(v(s)))(y))(\mu'*\beta(v(s)))(x)$ to $\abs{\mathcal{V}_y(s,y)-\mathcal{V}_x(s,x)}$, we get
\begin{align*}
    \abs{\mathcal{V}_y(s,y)-\mathcal{V}_x(s,x)}&\le 
   \abs{ \nu'((\mu*\beta(v))(s,y))((\mu'*\beta(v))(s,y)-(\mu'*\beta(v))(s,x))}\\
    &\quad+\abs{(\nu'((\mu*\beta(v(s)))(y))-\nu'((\mu*\beta(v(s)))(x)))(\mu'*\beta(v(s)))(x)}.
\end{align*}
% Now, 
% \begin{eqnarray*}
% &&\nn|\mu'*\beta(v)(s,y)-\mu'*\beta(v)(s,x)|\\
% % &=&\abs{\int_{\R}\beta(v(s,z))(\mu'(y-z)-\mu'(x-z))\d z}\\
% &\le&\abs{\beta}_{\lip(\R)}\abs{\mu'}_{\lip(\R)}\norma{v(s,\dott)}_{L^1(\R)}\abs{y-x}. 
% \end{eqnarray*}

% Similarly,
% \begin{eqnarray*}
% &&\nn|\nu'(\mu*\beta(v))(s,y)-\nu'(\mu*\beta(v))(s,x)|\\&\le&\abs{\nu'}_{\lip(\R)}\abs{(\mu*\beta(v))(s,y)-(\mu*\beta(v))(s,x)}\\
% &=&\abs{\nu'}_{\lip(\R)}\abs{\int_{\R}\beta(v(s,z))(\mu(y-z)-\mu(x-z))\d z}\\
% &=&\abs{\nu'}_{\lip(\R)}\abs{\beta}_{\lip(\R)}\abs{\mu}_{\lip(\R)}\norma{v(s,\dott)}_{L^1(\R)}\abs{y-x}. 
% \end{eqnarray*}
Furthermore, 
\begin{align*}
    |\mu'*\beta(v(s))(y)-\mu'*\beta(v(s))(x)|&\le \abs{\beta}_{\lip(\R)}\abs{\mu'}_{\lip(\R)}\norma{v(s,\dott)}_{L^1(\R)}\abs{y-x},\\
    |\nu'(\mu*\beta(v(s)))(y)-\nu'(\mu*\beta(v(s)))(x)|&\le\abs{\nu'}_{\lip(\R)}\abs{\beta}_{\lip(\R)}\abs{\mu}_{\lip(\R)}\norma{v(s,\dott)}_{L^1(\R)}\abs{y-x},
\end{align*}
which implies that
\begin{align*}
&\abs{\mathcal{V}_y(s,y)-\mathcal{V}_x(s,x)}\\&\quad\le     \abs{\nu}_{\lip(\R)}\abs{\beta}_{\lip(\R)}\abs{\mu'}_{\lip(\R)}\norma{v(s,\dott)}_{L^1(\R)}\abs{y-x}\\    &\qquad+\abs{\nu'}_{\lip(\R)}\abs{\beta}_{\lip(\R)}\abs{\mu}_{\lip(\R)}\norma{v(s,\dott)}_{L^1(\R)}\abs{y-x}\norma{\mu'}_{L^\infty(\R)}\abs{\beta}_{\lip(\R)}\norma{v(s,\dott)}_{L^1(\R)}.
\end{align*}
 Adding and subtracting $\nu'((\mu*\beta((s)))(x))(\mu'*\beta(u(t)))(x)$ to 
 $\abs{\mathcal{V}_x(s,x)-\mathcal{U}_x(t,x)}$, we ge
 \begin{align*}
   \abs{\mathcal{V}_x(s,x)-\mathcal{U}_x(t,x)}&\le \abs{\nu'((\mu*\beta(v(s)))(x))((\mu'*\beta(v(s)))(x)-(\mu'*\beta(u(t)))(x))}\\
     &\quad+\abs{(\nu'((\mu*\beta(v(s)))(x))-\nu'((\mu*\beta(u(t)))(x)))(\mu'*\beta(u(t)))(x)}.
 \end{align*}
%  Now, 
% \begin{eqnarray*}
% &&\nn|(\mu'*\beta(v))(s,x)-(\mu'*\beta(u))(t,x)|\\&\le&\abs{(\mu'*\beta(v))(s,x)-(\mu'*\beta(u))(t,x)}\nonumber\\
% &=&
% \abs{\int_{\R}(\beta(v(s,z))-\beta(u(t,z)))\mu'(x-z)\d z}\nonumber\\
% &=& \norma{\mu'}_{L^{\infty}(\R)}
% \norma{v(s,\dott)-u(t,\dott)}_{L^1(\R)}. 
% \end{eqnarray*}
% Finally,
% \begin{eqnarray*}
% &=&\nn|\nu'(\mu*\beta(v))(s,x)-\nu'(\mu*\beta(u))(t,x)|\\&\le&\abs{\nu'}_{\lip(\R)}\abs{(\mu*\beta(v))(s,x)-(\mu*\beta(u))(t,x)}\nonumber\\
% &=&\abs{\nu'}_{\lip(\R)}
% \abs{\int_{\R}(\beta(v(s,z))-\beta(u(t,z)))\mu(x-z)\d z}\Big)\nonumber\\
% &=&\abs{\nu'}_{\lip(\R)}\abs{\beta}_{\lip(\R)} \norma{\mu}_{L^{\infty}(\R)}
% \norma{v(s,\dott)-u(t,\dott)}_{L^1(\R)}. 
% \end{eqnarray*}
Moreover, 
\begin{align*}
  |(\mu'*\beta(v))(s,x)-(\mu'*\beta(u))(t,x)|&\le   \norma{\mu'}_{L^{\infty}(\R)}
 \norma{v(s,\dott)-u(t,\dott)}_{L^1(\R)},\\
 |\nu'(\mu*\beta(v))(s,x)-\nu'(\mu*\beta(u))(t,x)|&\le \abs{\nu'}_{\lip(\R)}\abs{\beta}_{\lip(\R)} \norma{\mu}_{L^{\infty}(\R)} \norma{v(s,\dott)-u(t,\dott)}_{L^1(\R)}.
\end{align*}
Collecting all terms, we have
\begin{align*}
   & \abs{\mathcal{V}_x(s,x)-\mathcal{U}_x(t,x)}\\&\le \norma{\nu'}_{L^\infty(\R)}\norma{\mu'}_{L^{\infty}(\R)}
\norma{v(s,\dott)-u(t,\dott)}_{L^1(\R)}\\
&\quad+\abs{\nu'}_{\lip(\R)}\abs{\beta}_{\lip(\R)} \norma{\mu}_{L^{\infty}(\R)}
\norma{v(s,\dott)-u(t,\dott)}_{L^1(\R)}\norma{\mu'}_{L^\infty(\R)}\abs{\beta}_{\lip(\R)}\norma{v(s,\dott)}_{L^1(\R)}.
\end{align*}
Now, it can be observed that $I_{\mathcal{U}_x}$ can be handled like \eqref{IUf}, leading to the following estimate:
\begin{equation}
 I_{\mathcal{U}_x}\le
\mathcal{K}_4(\epsilon+\epsilon_0)+\mathcal{K}_5\int_{0}^T
\norma{v(s,\dott)-u(s,\dott)}_{L^1(\R)}\d{s},
\end{equation}
for some appropriate constants $\mathcal{K}_4$ and $\mathcal{K}_5$.
Substituting the above estimates in \eqref{kuz}, we get
% ****Check, $\epsilon$, see $v,u$****
%  \begin{eqnarray*}
% \norma{u(T)-v(T)}_{L^1(\R)}&\le&-\displaystyle\Lambda_{\epsilon,\epsilon_0}(v,u)+\epsilon\TV(v_0)+ \epsilon\abs{f}_{\lip(\R)}\TV(v_0)+ \norma{u_0-v_0}_{L^1(\R)}+\epsilon\TV(v_0)\\
% &&+\abs{f}_{\lip(\R)}\abs{\nu}_{\lip(\R)}\abs{\beta}_{\lip(\R)}(\abs{u}_{BV(Q_T)}\abs{\mu}_{\lip(\R)}+\abs{f}_{\lip(\R)}\norma{u}_{L^1(\R)}\abs{\mu'}_{\lip(\R)})T\epsilon\int_0^T\norma{v(s,\dott)}_{L^1(\R)}\d{s}\\
% &&+\abs{f}_{\lip(\R)}\abs{\nu}_{\lip(\R)}\abs{\beta}_{\lip(\R)}(\abs{u}_{BV(Q_T)}+\abs{f}_{\lip(\R)}\norma{u}_{L^1(\R)})\Big(\gamma(u,{\epsilon_0})T
% +\int_{0}^T
% \norma{v(s,\dott)-u(s,\dott)}_{L^1(\R)}\d{s}\Big)\\
% % &=&\norma{u_0-v_0}_{L^1(\R)}+\alpha(T)+\beta\int_{0}^T
% % \norma{v(s,\dott)-u(s,\dott)}_{L^1(\R)}\d{s},
%  \end{eqnarray*}
% which is of the form
\begin{align*}
\norma{u(T,\dott)-v(T,\dott)}_{L^1(\R)}&\le-\displaystyle\Lambda_{\epsilon,\epsilon_0}(v,u)+\norma{u_0-v_0}_{L^1(\R)}+\mathcal{K}_6(\epsilon+\epsilon_0+\gamma(v,\epsilon_0))\\
&\quad+\mathcal{K}_7\int_{0}^T
\norma{v(s,\dott)-u(s,\dott)}_{L^1(\R)}\d{s}.
\end{align*}
Now, the result follows by Gronwall's inequality.
% , we get:
% \begin{eqnarray*}
% \norma{u(T)-v(T)}_{L^1(\R)} \le C^T\left[\displaystyle\Lambda_{\epsilon,\epsilon_0}(v,u) +\norma{u_0-v_0}_{L^1(\R)}+\epsilon + \gamma(v,\epsilon_0) \right]
% \end{eqnarray*}
% ***Caution:Note that Grownwall can be applied provided $C_2\geq 0,$ which is not straght forwards in this case as $\Lambda^{\epsilon,\epsilon_0}$ can be negative...justify that even if $\Lambda^{\epsilon,\epsilon_0}$ becomes negative $C_2$ remains positive.
\end{proof}
The remaining section is  dedicated to estimating the relative entropy functional $\Lambda_{\epsilon,\epsilon_0}(u^{\Delta},u)$ for which 
we follow the following notations: \\
For $i\in \Z, n\in \N, k\in \R,(t,x)\in Q_T,$ define
\begin{enumerate}
    \item $\eta_{i}^n(k):=\abs{u_i^n-k}$,
    \item $C_i^n:=C^n\times C_i$,
    \item $p_i^n(k):=G(u_i^n,k)=\sgn(u_i^n-k) (f(u_i^n)-f(k))$,
    \item $\mathcal{U}^{\Delta}(t,x):= \nu(\mu*\beta(u^{\Delta}(t)))(x)$.
\end{enumerate}

\begin{lemma}\label{lemma:est} The relative entropy functional $\Lambda_{\epsilon,\epsilon_0} (u^{\Delta},u)$ satisfies:
$$
-\Lambda_{\epsilon,\epsilon_0} (u^{\Delta},u)\leq \mathcal{C}\left(\frac{\D x}{\epsilon}+\frac{\D t}{\epsilon_0}\right),
$$
where $\mathcal{C}$ is a constant independent of $\Delta x, \Delta t.$
\end{lemma}
\begin{proof}
% Since
% $u^{\Delta}(t,x) =\sum_{n=0}^{N-1} \sum_{i\in \mathbb{Z}} \chi_i(x) \chi^n(t) u_i^{n}$ {\color{blue} for $(t,x)\in [0,T)\times \R$ and 
%  $u^{\D}(T,x)=\sum\limits_{i\in \Z}\chi_i(x) u_i^N$},
 
Let $\sum_{i,n}$ denote the double summation $\sum_{i\in \Z}\sum_{n=0}^{N-1}.$ For the piecewise constant function $u^{\D}$ (cf.\ \eqref{scheme2}-\eqref{def_u_De}) the relative entropy can be written as
\begin{align*}
&-\Lambda_{\epsilon,\epsilon_0}(u^{\Delta},u) \\
 &\quad= -\int_{Q_T}\displaystyle \displaystyle\sum_{i,n}   \int_{C_i^n}\eta_{i}^n(u(s,y))\Phi_t(s,y,t,x) \d t  \d x \d s \d y \\
 &\qquad-\int_{Q_T}\displaystyle \displaystyle\sum_{i,n}   \int_{C_i^n} p_i^n(u(s,y))\mathcal{U}^{\Delta}(t,x)\Phi_x(s,y,t,x) \d t  \d x \d s \d y \\
 &\qquad + \int_{Q_T}\displaystyle\sum_{i,n}  \int_{C_i^n}\sgn(u_i^n-u(s,y))f(u(s,y))\partial_x\mathcal{U}^{\Delta}(t,x)\Phi(s,y,t,x) \d t  \d x \d s \d y \\
    &\qquad- \int_{Q_T}\sum_{i}   \int\limits_{C_i}  \eta_{i}^0(u(s,y)) \Phi(s,y,0,x)  \d x\d s  \d y +\int_{Q_T} \sum_{i}  \int\limits_{C_i}  \eta_{i}^N(u(s,y))  \Phi(s,y,T,x)  \d x\d s  \d y.
\end{align*}
 Applying the fundamental theorem of calculus, followed by summation by parts, we get 
\begin{align}
-\Lambda_{\epsilon,\epsilon_0}&(u^{\Delta},u) \nonumber\\
% &=&\int_{Q_T}\displaystyle \displaystyle\sum_{i,n}  \left(\eta_{i}^{n+1}(u(s,y))-\eta_{i}^{n}(u(s,y))\right)\int\limits_{C_i}\Phi(s,y,t^{n+1},x)  \d x\d s  \d y\nonumber\\
%  &&-\int_{Q_T}\displaystyle \displaystyle\sum_{i,n}   \int_{C_i^n} p_i^n \nu((u^{\D}*\mu)  (s,x))\Phi_x(s,y,t,x) \d t  \d x \d s \d y \nonumber\\
%  && + \int_{Q_T}\displaystyle\sum_{i,n}  \int_{C_i^n}\sgn(u_i^n-u(s,y))f(u(s,y))\partial_x\nu((u^{\Delta}*\mu)  (s,x))\Phi(s,y,t,x) \d t  \d x \d s \d y \\%&&\text{Since $u^{\Delta}(t,x)=u_i^n$ on $C_i\times C^n,$}\\
&\quad=\int_{Q_T}\displaystyle \displaystyle\sum_{i,n} \left(\eta_{i}^{n+1}(u(s,y))-\eta_{i}^{n}(u(s,y))\right)\int\limits_{C_i}\Phi(s,y,t^{n+1},x)  \d x\d s  \d y \nonumber\\
 &\qquad-\int_{Q_T}\displaystyle \displaystyle\sum_{i,n}   \int_{C_i^n}p_i^n(u(s,y))\mathcal{U}^{\Delta}(t,x)\Phi_x(s,y,t,x) \d t  \d x \d s \d y \nonumber\\
 & \qquad+ \int_{Q_T}\displaystyle\sum_{i,n}  \int_{C_i^n}\sgn(u_i^n-u(s,y))f(u(s,y))\mathcal{U}_x^{\Delta}(t,x)\Phi(s,y,t,x) \d t  \d x \d s \d y \nonumber\\
 &\quad:=\lambda_1+\lambda_2+\lambda_3. \label{CFL_LF.6a}
 \end{align}
% *************************
% Consider the term 
% \begin{eqnarray*}
%     \lambda_2&=&-\int_{Q_T}\displaystyle\sum_{i,n}  \int_{C_i^n}p_i^n(u(s,y))\nu(c_{i+1/2}^n)\Phi_x(s,y,t,x)  \d t\d x\d s  \d y +\mathcal{O}\left(\frac{\Delta x}{\epsilon}\right)
% \end{eqnarray*}
% \textbf{Claim-2}
% \begin{eqnarray*}
% \lambda_3&=&\int_{Q_T}\displaystyle\sum_{i,n}  \sgn(u_i^{n}-u(s,y))f(u(s,y)) (\nu(c_{i+\frac{1}{2}}^n)-\nu(c_{i-\frac{1}{2}}^n))\int\limits_{C^{n}} \Phi(t,x_i,y,s)  \d s \d t \d y+\mathcal{O}\left(\frac{\D x}{\epsilon}\right)\end{eqnarray*}

% \newpage
% ***********************\\
We have
\begin{align}
    \lambda_2&=-\int_{Q_T}\displaystyle\sum_{i,n}  \int_{C_i^n}p_i^n(u(s,y))\nu(c_{i+1/2}^n)\Phi_x(s,y,t,x)  \d t\d x\d s  \d y\nonumber\\ 
&\quad-\int_{Q_T}\displaystyle\sum_{i,n}  \int_{C_i^n}p_i^n(u(s,y))\Phi_x(s,y,t,x) \left(\mathcal{U}^{\Delta}(t,x)-\nu(c_{i+1/2}^n)\right)  \d t\d x\d s  \d y \nonumber\\
&:=\lambda_{2}'+\mathcal{E}_2. \label{CFL_LF.6b}
\end{align}
For $(n,i)\in \N\times \Z$ and $(t,x)\in C_i^n$
observe that
\begin{align*}
&\left|\mathcal{U}^{\Delta}(t,x)-\nu(c_{i+1/2}^n)\right|\\
&\qquad
\le \left|\mathcal{U}^{\Delta}(t,x)-\mathcal{U}^{\Delta}(t^n,x)\right| +\left|\mathcal{U}^{\Delta}(t^n,x)-\mathcal{U}^{\Delta}(t^n,x_{i+1/2})\right|+\left|\mathcal{U}^{\Delta}(t^n,x_{i+1/2})-\nu(c_{i+1/2}^n)\right|. 
\end{align*}
Also for $t \in C^n,$ using \eqref{apx:time}, we have
\begin{equation}
\abs{\mathcal{U}^{\Delta}(t^n,x)-\mathcal{U}^{\Delta}(t,x)}\le \abs{\nu}_{\lip(\R)}\abs{\beta}_{\lip(\R)}\norma{\mu}_{\L\infty(\R)}\mathcal{L}_3\Delta t. 
\end{equation}
Now using the Lipschitz continuity of $\mathcal{U}^{\Delta}$ in the space variable, we have
\begin{equation*}
\abs{\mathcal{U}^{\Delta}(t^n,x)-\mathcal{U}^{\Delta}(t^n,x_{i+1/2})}\leq \mathcal{C}_1\Delta x.
\end{equation*}
Furthermore, 
\begin{align}  &\abs{\mathcal{U}^{\Delta}(t^n,x_{i+1/2})-\nu(c_{i+1/2}^n)}\\
  % &\le& \abs{\nu}_{\lip(\R)}\abs{(u^{\Delta}*\mu)(x_{i+1/2},t^n)-c_{i+1/2}^n}\\
  % &=&\abs{\nu}_{\lip(\R)}\abs{\int\limits_{\R} \mu(x_{i+1/2}-y)u^{\D}(t^n,y)\d y-\Delta x\sum\limits_{p\in Z} \mu_{i+1/2-p}u^{n}_{p}}\\  
   % \mbox{Using any first order Quadrature Rule}
&\quad\leq\abs{\nu}_{\lip(\R)}\abs{\beta}_{\lip(\R)}\abs{\sum_{p\in\Z}\int\limits_{C_p} \mu(x_{i+1/2}-y)u^n_p\d y-\Delta x\sum\limits_{p\in \Z} \mu(x_{i+1/2}-x_{p})u^{n}_{p}}\nonumber\\
&\quad=\abs{\nu}_{\lip(\R)}\abs{\beta}_{\lip(\R)}\abs{\sum_{p\in\Z} u^n_p\left(\int\limits_{C_p} \mu(x_{i+1/2}-y)\d y-\Delta x\mu(x_{i+1/2}-x_{p})\right)}\nonumber\\  &\quad=\abs{\nu}_{\lip(\R)}\abs{\beta}_{\lip(\R)}\norma{u_0^{\Delta}}_{L^1(\R)} \abs{\mu}_{\lip(\R)}\Delta x:=\mathcal{C}_1\Delta x \label{nn}.
  \end{align}  

% Combining the above estimates, for $x\in C_i$ we get:
% \begin{eqnarray}
% \abs{\mathcal{U}^{\Delta}(t^n,x)-\nu(c_{i+1/2}^n)}\le 2\abs{\nu}_{\lip(\R)}\abs{\beta}_{\lip(\R)}\norma{u^{\Delta}}_{L^1(Q_T)} \abs{\mu}_{\lip(\R)}\Delta x
% \end{eqnarray}
% Since $\mu$ is a bounded function, we also have the following time estimate:
% \begin{eqnarray}
% \abs{u^{\D}*\mu(t,x_1)-u^{\D}*\mu(t,x_2)} &\leq& \int\limits_{\R} \abs{u^{\D}(x-y,t_1)\mu(y)-u^{\D}(x-y,t_2)\mu(y)} dy\\ &\leq& \norma{\mu}_{L^{\infty}(\R)} \norma{u^{\D}(\cdot,t_1)-u^{\D}(\cdot,t_2)}_{L^1(\R)}
% \end{eqnarray}
% ********Estimates are not clean to be revised************\\
% Now, using the time continuity of $u^{\Delta}$, on every $C_n$, we have
Finally, combining all the above estimates we get, 
\begin{align*}
\nonumber\abs{\mathcal{E}_2}&\le
%\sum_{i,n} 
 \int_{Q_T}\displaystyle\sum_{i,n}  \mathcal{C}_1\Delta x \abs{f(u_i^n)-f(u(s,y))}\int_{C_i^n}  \abs{\Phi_x(s,y,t,x)}   \d t\d x\d s  \d y\\\nonumber
&\leq\int_{Q_T}\displaystyle\sum_{i,n}  \mathcal{C}_1\Delta x (\abs{f(u_i^n)}+\abs{f(u(s,y))})\int_{C_i^n} \abs{\Phi_x(s,y,t,x)}   \d t\d x\d s  \d y\\
&=\displaystyle\sum_{i,n} \int_{  C_i^n}  \mathcal{C}_1\Delta x (\abs{f(u_i^n)}\int_{Q_T} \abs{\Phi_x(s,y,t,x)}   \d t\d x\d s  \d y\\
&+\int_{Q^2_T}\displaystyle  \mathcal{C}_1\Delta x\abs{f(u(s,y))}\abs{\Phi_x(s,y,t,x)}   \d t\d x\d s  \d y\\
&=\mathcal{C}_1\abs{f}_{\lip(\R)}\norma{u^{\D}}_{L^1(Q_T)}\frac{\Delta x}{\epsilon}+\displaystyle  \mathcal{C}_1\abs{f}_{\lip(\R)}\norma{u}_{L^1(Q_T)}\frac{\Delta x}{\epsilon}.\\
    % &=&-\displaystyle\sum_{i,n}  \int_{C_i^n}\left(\abs{f(u_{i+1}^n)-f(u(s,y))}- \abs{f(u_{i}^n)-f(u(s,y))}\right)C\Delta x \Phi(s,x)\d s  \d x
\end{align*}
Thus, we have
\begin{equation}\label{L2}
    \lambda_2= \lambda_2'+\mathcal{O}\left(\frac{\Delta x}{\epsilon}\right).
\end{equation}
Now, we consider 
\begin{align}
    \lambda_3&= \int_{Q_T}\displaystyle\sum_{i,n}  \sgn(u_i^{n}-u(s,y))f(u(s,y)) (\nu(c_{i+\frac{1}{2}}^n)-\nu(c_{i-\frac{1}{2}}^n))\int\limits_{C^{n}} \Phi(s,y,t,x_{i+1/2})  \d t \d s \d y \nonumber\\
    &+\int_{Q_T}\displaystyle\sum_{i,n}  \sgn(u_i^{n}-u(s,y))f(u(s,y))\nonumber \\
    & \quad \quad  \times 
    \int_{C_i^n}\mathcal{U}_x^{\Delta}(t,x)  \left(\Phi(s,y,t,x)-\Phi(s,y,t,x_{i+1/2})\right) \d x \d t  \d y \d s \nonumber\\
    &+ \int_{Q_T}\displaystyle\sum_{i,n}  \sgn(u_i^{n}-u(s,y))f(u(s,y))\nonumber \\ 
    & \quad \quad \times 
    \int_{C_i^n}\left(\mathcal{U}_x^{\Delta}(t,x)  -\frac1{\Delta x}\big(\nu(c_{i+\frac{1}{2}}^n)-\nu(c_{i-\frac{1}{2}}^n)\big)\right)\Phi(s,y,t,x_{i+1/2}) \d x \d t \d y \d s\nonumber\\
    &\quad :=\lambda_3'+\mathcal{E}_{31}+\mathcal{E}_{32}. \label{CFL_LF.9a}
\end{align}
The error terms can be estimated as below,
\begin{align*}
    \abs{\mathcal{E}_{31}} &\leq  \int_{Q_T}\displaystyle\sum_{i,n}  \abs{f(u(s,y))}
    \int_{C_i^n} \abs{\mathcal{U}_x^{\Delta}(t,x) }  \abs{ \Phi(s,y,t,x)-\Phi(s,y,t,x_{i+1/2})} \d t \d x \d s \d y \\
    &\leq
|\nu|_{\lip(\R)}|\beta|_{\lip(\R)}|\mu|_{\lip(\R)} \norma{u_0}_{L^1(\R)}\\&\qquad \times\int_{Q_T}\displaystyle\sum_{i,n}  \abs{f(u(s,y))}
    \int_{C_i^n}   \abs{ \Phi(s,y,t,x)-\Phi(s,y,t,x_{i+1/2})} \d x \d t  \d y \d s\\
    &\leq  |\nu|_{\lip(\R)}|\beta|_{\lip(\R)}|\mu|_{\lip(\R)} \norma{u_0}_{L^1(\R)}\\&\qquad \times\int_{Q_T}\displaystyle \abs{f(u(s,y))}
       \sum_{i}\int\limits_{C_i}     \abs{\omega_{\epsilon} (y-x)-\omega_{\epsilon}(y-x_{i+1/2})} \d x \d s  \d y
    \\
    &\leq  |\nu|_{\lip(\R)}|\beta|_{\lip(\R)}|\mu|_{\lip(\R)} \norma{u_0}_{L^1(\R)}|w^{\epsilon}|_{BV(\R)} \D x \int_{Q_T}\displaystyle \abs{f(u(s,y))}
     \d s  \d y \\
     &\leq |\nu|_{\lip(\R)}|\beta|_{\lip(\R)}|\mu|_{\lip(\R)} \norma{u_0}_{L^1(\R)} \norma{f(u)}_{L^1(Q_T)}\frac{\D x}{\epsilon},
\end{align*}
using Theorem \ref{thm:2.2}, equation \eqref{apx:L1}, and 
\begin{equation*}
|\mathcal{U}_x^{\Delta}(t,x)|=|\nu'((\beta(u^{\Delta})*\mu)(t,x))(\mu'*\beta(u^\Delta))(t,x)|\leq |\nu|_{\lip(\R)}|\beta|_{\lip(\R)}|\mu|_{\lip(\R)} \norma{u_0}_{L^1(\R)}.
% &\le &\norma{\partial_x\mu*v}_{\infty}\norma{D^2\nu}_{\infty}\norma{\mu}_{\infty}\norma{u(.,s)-v(.,s)}_{\L1(\mathbb{R})}+\abs{\nu}_{\lip(\R)}\norma{\eta_x}_{\infty}\norma{u(.,s)-v(.,s)}_{\L1(\mathbb{R})}
\end{equation*}
Furthermore, we have, by applying 
fundamental theorem of calculus and rearrange the terms, 
\begin{align*}
\abs{\mathcal{E}_{32} }&=\Big|\int_{Q_T} \sum_{i,n}  \sgn(u_i^{n}-u(s,y))f(u(s,y))\int\limits_{C^n}\Phi(s,y,t,x_{i+1/2}) \\
&\qquad \times
   \int\limits_{C_i} \big( \mathcal{U}_x^{\Delta}(t,x)-\frac1{\Delta x}(\nu(c_{i+\frac{1}{2}}^n)-\nu(c_{i-\frac{1}{2}}^n))\big)\d x \d t \d y \d s \Big| \\
&=\Big|\int_{Q_T} \sum_{i,n} \sgn(u_i^{n}-u(s,y)) f(u(s,y))\\
&\qquad  \times{\left(
   \mathcal{U}^{\Delta}(t^n,x_{i+1/2})-\mathcal{U}^{\Delta}(t^n,x_{i-1/2})-\nu(c_{i+\frac{1}{2}}^n)+\nu(c_{i-\frac{1}{2}}^n)\right)\int\limits_{C^n}\Phi(s,y,t,x_{i+1/2}) \d s \d t \d y} 
   \\
&\quad+{\int_{Q_T}\displaystyle\sum_{i,n} \sgn(u_i^{n}-u(s,y)) f(u(s,y))\int\limits_{C^n}\Phi(s,y,t,x_{i+1/2})} \\
& \qquad  \times{\big(
  \mathcal{U}^{\Delta}(t,x_{i+1/2})- \mathcal{U}^{\Delta}(t^n,x_{i+1/2})  -\mathcal{U}^{\Delta}(t,x_{i-1/2})+\mathcal{U}^{\Delta}(t^n,x_{i-1/2})\big) \d t \d s \d y} \Big|.
\end{align*}
Apply summation by parts in $i$ to get,
\begin{align*}
\abs{\mathcal{E}_{32} }&=\Big|{\int_{Q_T}\displaystyle\sum_{i,n} \sgn(u_i^{n}-u(s,y)) f(u(s,y))}\\
& \qquad \times{\left(
   \mathcal{U}^{\Delta}(t^n,x_{i+1/2})-\nu(c_{i+\frac{1}{2}}^n)\right)\int\limits_{C^n}(\Phi(s,y,t,x_{i+1/2})-\Phi(s,y,t,x_{i-1/2})) \d t \d s \d y} \Big|
   \\
&\quad+\Big|{\int_{Q_T}\displaystyle\sum_{i,n} \sgn(u_i^{n}-u(s,y)) f(u(s,y))}\\
& \qquad \times {\int\limits_{C^n} \big(
  \mathcal{U}^{\Delta}(t,x_{i+1/2})- \mathcal{U}^{\Delta}(t^n,x_{i+1/2})\big) (\Phi(s,y,t,x_{i+1/2})-\Phi(s,y,t,x_{i-1/2}))\d t \d s \d y}\big| \\
  &:=\mathcal{E}_{321}+\mathcal{E}_{322}.
  \end{align*}
Now,  using \eqref{nn}, we have
  \begin{align*}
\abs{\mathcal{E}_{321}}
&\le{\int_{Q_T} \Big(\sum_n\abs{f(u(s,y))}\abs{\nu}_{\lip(\R)}\abs{\beta}_{\lip(\R)}\norma{u^{\Delta}}_{L^1(Q_T)} \abs{\mu}_{\lip(\R)}\Delta x}\\
&  \qquad \times \int\limits_{C^n}\displaystyle\sum_{i}\abs{\Phi(s,y,t,x_{i+1/2})-\Phi(s,y,t,x_{i-1/2})} \Big) \d t \d s \d y \\
&\leq{\int_{Q_T} \abs{f(u(s,y))} \abs{\nu}_{\lip(\R)}\abs{\beta}_{\lip(\R)}\norma{u^{\Delta}}_{L^1(Q_T)} \abs{\mu}_{\lip(\R)}\Delta x|\omega_{\epsilon}|_{BV(\R)} \d s \d y}
   \\
   &\le{\int_{Q_T} \abs{f(u(s,y))} \abs{\nu}_{\lip(\R)}\abs{\beta}_{\lip(\R)}\norma{u^{\Delta}}_{L^1(Q_T)} \abs{\mu}_{\lip(\R)}\Delta x \frac{1}{\epsilon} \d s  \d y}\\
   &= {\abs{\nu}_{\lip(\R)}\abs{\beta}_{\lip(\R)}\norma{u^{\Delta}}_{L^1(Q_T)} \abs{\mu}_{\lip(\R)} \abs{f}_{\lip(\R)}\norma{u}_{L^1(Q_T)}\frac{\Delta x}{\epsilon}}.
 \end{align*}
 Using \eqref{apx:time}, for $t\in C^n,$ we have 
 \begin{align*}
&\abs{\mathcal{U}^{\Delta}(t,x_{i+1/2})- \mathcal{U}^{\Delta}(t^n,x_{i+1/2})} \\&\quad\leq \abs{\nu}_{\lip(\R)}\abs{\beta}_{\lip(\R)} \norma{\mu}_{L^{\infty}(\R)}\int\limits_{\R} \abs{u^{\D}(t,x_{i+1/2}-y)\mu(y)-u^{\D}(t^n,x_{i+1/2}-y)\mu(y)} dy\\ 
&\quad= \abs{\nu}_{\lip(\R)}\abs{\beta}_{\lip(\R)}\norma{\mu}_{L^{\infty}(\R)} \norma{u^{\D}(t,\dott)-u^{\D}(t^n,\dott)}_{L^1(\R)}\\
&\quad \le  \abs{\nu}_{\lip(\R)}\abs{\beta}_{\lip(\R)}\norma{\mu}_{L^{\infty}(\R)} \mathcal{L}_3\Delta t.
\end{align*}
Consequently,
  \begin{align*}
\abs{\mathcal{E}_{322}}
&\le\int_{Q_T}\sum_n\abs{f(u(s,y))}\abs{\nu}_{\lip(\R)}\abs{\beta}_{\lip(\R)}\norma{\mu}_{L^{\infty}(\R)} \mathcal{L}_3\Delta t\\
&\qquad  \times \int\limits_{C^n}\displaystyle\sum_{i}\abs{\Phi(s,y,t,x_{i+1/2})-\Phi(s,y,t,x_{i-1/2})} \d t \d s \d y \\
&\leq{\int_{Q_T} \abs{f(u(s,y))} \abs{\nu}_{\lip(\R)}\abs{\beta}_{\lip(\R)}\norma{\mu}_{L^{\infty}(\R)} \mathcal{L}_3\Delta t|w^{\epsilon}|_{BV(\R)}\int_{0}^T|\omega_{\epsilon_0}(s-t)| \d t \d s \d y}
   \\
   &\le {\int_{Q_T} \abs{f(u(s,y))} \abs{\nu}_{\lip(\R)}\abs{\beta}_{\lip(\R)}\norma{\mu}_{L^{\infty}(\R)} \mathcal{L}_3\Delta t \frac{1}{\epsilon} \d s  \d y}\\
   &= \abs{\nu}_{\lip(\R)}\abs{\beta}_{\lip(\R)}\norma{\mu}_{L^{\infty}(\R)} \mathcal{L}_3 \abs{f}_{\lip(\R)} \norma{u}_{L^1(Q_T)}\frac{\Delta t}{\epsilon}.
 \end{align*}

% \fbox{Why not $\norma{f(u)}_{L^1(Q_T)}\le \abs{f}_{\lip(\R)}\norma{u}_{L^1(Q_T)}$?}

Finally, combining the above estimates, we get:
\begin{equation}
\lambda_3 = \lambda_3' +\mathcal{O}\left(\frac{\D x}{\epsilon}\right).
\end{equation}
Thus, so far we have proved
\begin{equation}\label{est:5}
-\Lambda_{\epsilon,\epsilon_0}(u^{\Delta},u)=\lambda_1+\lambda_2'+\lambda_3'+\mathcal{O}\left(\frac{\D x}{\epsilon}\right).
\end{equation}
Recall $\lambda_1$, cf.\ \eqref{CFL_LF.6a}, 
\begin{align*}
\lambda_1&=\int_{Q_T}\displaystyle \displaystyle\sum_{i,n}  \left(\eta_{i}^{n+1}(u(s,y))-\eta_{i}^{n}(u(s,y))\right)\int\limits_{C_i}\Phi(s,y,t^{n+1},x)  \d x\d s  \d y  \\    
&\le -\lambda \int_{Q_T}\displaystyle \displaystyle\sum_{i,n}  \big(
\mathcal{G}^{n}_{i+1/2}(u_i^{n} ,u_{i+1}^{n},u(s,y) )-\mathcal{G}^{n}_{i-1/2}(u_{i-1}^{n} ,u_i^{n},u(s,y) )\big) \int\limits_{C_i}\Phi(s,y,t^{n+1},x)  \d x\d s  \d y\\
 &\quad -\lambda \int_{Q_T}\displaystyle \displaystyle\sum_{i,n}\sgn(u_i^{n+1}-u(s,y))f(u(s,y) )(\nu(c_{i+\frac{1}{2}}^{n})-\nu(c_{i-\frac{1}{2}}^{n}))
\int\limits_{C_i}\Phi(s,y,t^{n+1},x)  \d x\d s  \d y\\
&:= A_2+A_3,
\end{align*}
by applying the discrete entropy inequality \eqref{apx:de}.

Furthermore, we can rewrite $\lambda_2'$, cf.\ \eqref{CFL_LF.6b}, as follows
\begin{align*}
\lambda_2'&= -\int_{Q_T}\displaystyle\sum_{i,n}  \int_{C_i^n}p_i^n(u(s,y))\nu(c_{i+1/2}^n)\Phi_x(s,y,t,x)  \d t\d x\d s  \d y\\
&=\int_{Q_T}\displaystyle\sum_{i,n}    \int\limits_{C^{n}} \left(p_i^n(u(s,y))\nu(c_{i+1/2}^n)- p_{i-1}^n(u(s,y))\nu(c_{i-1/2}^n) \right)\Phi(s,y,t,x_{i+1/2})  \d t\d s  \d y
\end{align*}
by using the fundamental theorem of calculus, followed by summation by parts.

\begin{enumerate}[label=\textbf{Claim 1}]
	\item \label{C1} $A_2+\lambda_2' =\mathcal{O}\left(\frac{\D x}{\epsilon}+\frac{\D t}{\epsilon_0}\right)$.
\end{enumerate}
Adding and subtracting 
$$
\lambda\int_{Q_T}\displaystyle\sum_{i,n}\big(p_i^n(u(s,y))\nu(c_{i+1/2}^n)-p_{i-1}^n(u(s,y))\nu(c_{i-1/2}^n)\big)\int\limits_{C_i}\Phi(s,y,t^{n+1},x)\d x\d y \d s
$$  
we have
\begin{align}
A_2&+\lambda'_2\\&=-\lambda\int_{Q_T}\big(\mathcal{G}^{n}_{i+1/2}(u_i^{n} ,u_{i+1}^{n},u(s,y) )-\mathcal{G}^{n}_{i-1/2}(u_{i-1}^{n} ,u_i^{n},u(s,y))\big)\int\limits_{C_i}\Phi(s,y,t^{n+1},x)\d x\d y \d s\nonumber\\
&\quad+\int_{Q_T}\displaystyle\sum_{i,n}\big(p_i^n(u(s,y))\nu(c_{i+1/2}^n)-p_{i-1}^n(u(s,y)\big)\nu(c_{i-1/2}^n))\nonumber\\
&\qquad \times\big(\int\limits_{C^{n}} \Phi(s,y,t,x_{i+\frac{1}{2}})\d t-\lambda\int\limits_{C_i}\Phi(s,y,t^{n+1},x)\d x\big)\d y \d s\nonumber\\
&\quad+\lambda\int_{Q_T}\displaystyle\sum_{i,n}\big(p_i^n(u(s,y))\nu(c_{i+1/2}^n)-p_{i-1}^n(u(s,y))\nu(c_{i-1/2}^n)\big)\int\limits_{C_i}\Phi(s,y,t^{n+1},x)\d x\d y \d s\nonumber\\
\nonumber\\
&=\lambda\int_{Q_T}\displaystyle\sum_{i,n}\big(p_i^n(u(s,y))\nu(c_{i+1/2}^n)-\mathcal{G}^{n}_{i+1/2}(u_i^{n} ,u_{i+1}^{n},u(s,y) )\big)\int\limits_{C_i}\Phi(s,y,t^{n+1},x)\d x\d y \d s\nonumber\\
&\quad-\lambda\int_{Q_T}\displaystyle\sum_{i,n}\big(p_{i-1}^n(u(s,y))\nu(c_{i-1/2}^n)-\mathcal{G}^{n}_{i-1/2}(u_{i-1}^{n} ,u_i^{n},u(s,y))\big)\int\limits_{C_i}\Phi(s,y,t^{n+1},x)\d x\d y \d s\nonumber\\
&\quad+\int_{Q_T}\displaystyle\sum_{i,n}\big(p_i^n(u(s,y))\nu(c_{i+1/2}^n)-p_{i-1}^n(u(s,y))\nu(c_{i-1/2}^n))\big)\nonumber\\
&\qquad \times\big(\int\limits_{C^{n}} \Phi(s,y,t,x_{i+\frac{1}{2}})\d t-\lambda\int\limits_{C_i}\Phi(s,y,t^{n+1},x)\d x\big)\d y \d s\nonumber.
\end{align}
Apply summation by parts to get
\begin{align}
A_2+\lambda'_2&=\lambda\int_{Q_T}\displaystyle\sum_{i,n}\big(\mathcal{G}^{n}_{i+1/2}(u_i^{n} ,u_{i+1}^{n},u(s,y) )-p_i^n(u(s,y))\nu(c_{i+1/2}^n)\big)\nonumber\\
&\qquad\times\big(\int\limits_{C_{i+1}}\Phi(s,y,t^{n+1},x)\d x-\int\limits_{C_i}\Phi(s,y,t^{n+1},x)\d x\big)\d y \d s\nonumber\\
&\quad+\int_{Q_T}\displaystyle\sum_{i,n}\big(p_i^n(u(s,y))\nu(c_{i+1/2}^n)-p_{i-1}^n(u(s,y))\nu(c_{i-1/2}^n)\big)\nonumber\\
&\qquad\times\big(\int\limits_{C^{n}} \Phi(s,y,t,x_{i+\frac{1}{2}})\d t-\lambda\int\limits_{C_i}\Phi(s,y,t^{n+1},x)\d x\big)\d y \d s,\nonumber
\end{align}
where for \begin{align*}
\mathcal{C}_1&=\mathcal{C}_1(\abs{f}_{\lip(\R)},\norma{\nu}_{L^{\infty}(\R)})\\ \mathcal{C}_2&=\mathcal{C}_2\left(\abs{f}_{\lip(\R)},\norma{\nu}_{L^{\infty}(\R)}, \norma{u}_{L^{\infty}(\overline{Q}_T}), \norma{u^{\D}}_{L^{\infty}(\overline{Q}_T)}\right)\end{align*} the following estimates hold:
\begin{align}
\abs{\mathcal{G}^{n}_{i+1/2}(u^n_i,u^n_{i+1},u(s,y))-\nu(c^n_{i+1/2}) p_i^n(u(s,y))  }&\leq \mathcal{C}_1   \abs{u^n_{i+1}-u^n_i}, \label{est:1} \\
\abs{\nu(c^n_{i+1/2}) p_i^n(u(s,y))-\nu(c^n_{i-1/2})p_{i-1}^n(u(s,y))} & \leq \mathcal{C}_2 \big(\abs{\nu(c^n_{i+1/2})-\nu(c^n_{i-1/2})} \nn\\&\qquad+ \abs{u^n_{i+1}-u^n_i} \big).
\label{est:2}
\end{align}
Using \eqref{est:1}--\eqref{est:2}, we have
\begin{align*}
A_2+\lambda_2'  &\leq \mathcal{C}_1 \lambda\int_{Q_T}\displaystyle\sum_{i,n} \abs{u^n_{i+1}-u^n_i}\nonumber\\
&\qquad\times\big|\int\limits_{C_{i+1}}\Phi(s,y,t^{n+1},x)\d x-\int\limits_{C_i}\Phi(s,y,t^{n+1},x)\d x\big|\d y \d s\nonumber\\
&\quad+\mathcal{C}_2 \int_{Q_T}\displaystyle\sum_{i,n}\left[\abs{\nu(c^n_{i+1/2})-\nu(c^n_{i-1/2})} + \abs{u^n_{i+1}-u^n_i} \right]\nonumber\\
&\qquad\times\big|\int\limits_{C^{n}} \Phi(s,y,t,x_{i+\frac{1}{2}})\d t-\lambda\int\limits_{C_i}\Phi(s,y,t^{n+1},x)\d x\big|\d y \d s. \nonumber
\end{align*}

Now, since $\mu$ has bounded variation, the claim follows because of the following estimates, 
\begin{align}  
\sum_{i}\Big|\nu(c_{i+1/2}^n)-\nu(c_{i-1/2}^n)\Big|&\le\abs{\nu}_{\lip(\R)}\sum_{i,p}\Delta x\Big|\beta(u^n_{p+1/2})(\mu_{i+1/2-p}-\mu_{i-1/2-p})\Big|\nonumber\\
&\le \abs{\nu}_{\lip(\R)}\sum_{i,p}\Delta x\abs{\beta}_{\lip(\R)}\abs{u^n_{p+1/2}}\int\limits_{C_{i-p}}\abs{\mu'(x)}\d x\nonumber\\
&=\abs{\nu}_{\lip(\R)}\sum_{p}\Delta x\abs{\beta}_{\lip(\R)}\abs{u^n_{p+1/2}}\int\limits_{\R}\abs{\mu'(x)}\d x\nonumber\\
&\leq \abs{\nu}_{\lip(\R)}\abs{\beta}_{\lip(\R)}\norma{u_0}_{L^1(\R)}\abs{\mu}_{BV(\R)}:=\mathcal{C}_3\label{cc},
\end{align}
which is true because of \eqref{apx:L1} and the estimates
\begin{align}
\int_{Q_T} \Big(\int\limits_{C^{n}} \Phi(s,y,t,x_{i+\frac{1}{2}})\d t-\lambda\int\limits_{C_i}\Phi(s,y,t^{n+1},x)\d x \Big) \d y \d s &=\mathcal{O} \left(\frac{\D x^2}{\epsilon} +\frac{\D t^2}{\epsilon_0}\right),\label{est:3}\\
\int_{Q_T}\Big(\int\limits_{C_{i+1}}\Phi(s,y,t^{n+1},x)\d x-\int\limits_{C_i}\Phi(s,y,t^{n+1},x)\d x \Big) \d y \d s&= \mathcal{O}\left( \frac{\D x^2}{\epsilon}\right)\label{est:4},
\end{align}
the proofs of which can be found in \cite[Ex.~3.17]{holden2015front}.

\begin{enumerate}[label=\textbf{Claim 2}]
\item \label{C2}  $A_3+\lambda_3' =\mathcal{O}\left(\frac{\D x}{\epsilon}\right)$.
\end{enumerate} We find
\begin{align*}
A_3&+\lambda_3'\\
&= -\lambda \int_{Q_T}\displaystyle \displaystyle\sum_{i,n}\sgn(u_i^{n+1}-u(s,y))f(u(s,y) )(\nu(c_{i+\frac{1}{2}}^{n})-\nu(c_{i-\frac{1}{2}}^{n}))
\int\limits_{C_i}\Phi(s,y,t^{n+1},x)  \d x\d s  \d y\\
&\quad+\int_{Q_T}\displaystyle\sum_{i,n}  \sgn(u_i^{n}-u(s,y))f(u(s,y)) (\nu(c_{i+\frac{1}{2}}^n)-\nu(c_{i-\frac{1}{2}}^n))\int\limits_{C^{n}} \Phi(s,y,t,x_{i+1/2})  \d t \d s \d y\\
&= -\lambda \int_{Q_T}\displaystyle \displaystyle\sum_{i,n}\sgn(u_i^{n+1}-u(s,y))f(u(s,y) )(\nu(c_{i+\frac{1}{2}}^{n})-\nu(c_{i-\frac{1}{2}}^{n}))
\int\limits_{C_i}\Phi(s,y,t^{n+1},x)  \d x\d s  \d y\\
&\quad {+\lambda \int_{Q_T}\displaystyle \displaystyle\sum_{i,n}\sgn(u_i^{n}-u(s,y))f(u(s,y) )(\nu(c_{i+\frac{1}{2}}^{n})-\nu(c_{i-\frac{1}{2}}^{n}))
\int\limits_{C_i}\Phi(s,y,t^{n+1},x)  \d x\d s  \d y}\\
&\quad-\lambda \int_{Q_T}\displaystyle \displaystyle\sum_{i,n}\sgn(u_i^{n}-u(s,y))f(u(s,y) )(\nu(c_{i+\frac{1}{2}}^{n})-\nu(c_{i-\frac{1}{2}}^{n}))
\int\limits_{C_i}\Phi(s,y,t^{n+1},x)  \d x\d s  \d y\\
&\quad+\int_{Q_T}\displaystyle\sum_{i,n}  \sgn(u_i^{n}-u(s,y))f(u(s,y)) (\nu(c_{i+\frac{1}{2}}^n)-\nu(c_{i-\frac{1}{2}}^n))\int\limits_{C^{n}} \Phi(s,y,t,x_{i+1/2})  \d t \d s \d y\\
&=
\int_{Q_T} \displaystyle \displaystyle\sum_{i,n}\sgn(u_i^{n}-u(s,y))f(u(s,y) )(\nu(c_{i+\frac{1}{2}}^{n})-\nu(c_{i-\frac{1}{2}}^{n}))\\
&\qquad \qquad \times \big(\int\limits_{C^{n}} \Phi(s,y,t,x_{i+1/2})\d t )   -\lambda\int\limits_{C_i}\Phi(s,y,t^{n+1},x) \d x \big)  \d s \d y\\
&\quad-\lambda \int_{Q_T} \displaystyle \displaystyle\sum_{i,n}\big(\sgn(u_i^{n+1}-u(s,y))-\sgn(u_i^{n}-u(s,y))\big)\\
&\qquad \qquad  \times  f(u(s,y))(\nu(c_{i+\frac{1}{2}}^{n})-\nu(c_{i-\frac{1}{2}}^{n}))
\int\limits_{C_i}\Phi(s,y,t^{n+1},x)  \d x   \d s \d y\\
&:=\tilde{\mathcal{E}}_1+ \tilde{\mathcal{E}}_2.
\end{align*}
The terms $\tilde{\mathcal{E}}_1$ and $\tilde{\mathcal{E}}_2$ can be estimated as follows:
\begin{align*}
\abs{\tilde{\mathcal{E}}_1}
&\le \int_{Q_T}\displaystyle \displaystyle\sum_{i,n}|f(u(s,y))||\nu(c_{i+\frac{1}{2}}^{n})-\nu(c_{i-\frac{1}{2}}^{n})|
\\&\qquad \times\Big|\int\limits_{C^{n}} \Phi(s,y,t,x_{i+1/2})\d t -\lambda\int\limits_{C_i}\Phi(s,y,t^{n+1},x) \d x  \Big|  \d s \d y\\
&\le \norma{f(u)}_{L^{\infty}(Q_T)} \sum_{i,n}\Big|\nu(c_{i+1/2}^n)-\nu(c_{i-1/2}^n)\Big|\\&\qquad\times \int_{Q_T}\Big|\int\limits_{C^{n}} \Phi(s,y,t,x_{i+\frac{1}{2}})\d t-\lambda\int\limits_{C_i}\Phi(s,y,t^{n+1},x)\d x\Big|\d y \d s\\
&\le T\norma{f(u)}_{L^{\infty}(Q_T)}\mathcal{C}_3 \left(\frac{\Delta x}{\epsilon}+\frac{\Delta t}{\epsilon_0}\right),
\end{align*}
 using \eqref{cc}--\eqref{est:3}. 
Furthermore,  
% \fbox{Check summation limits: Is it $\sum_{n=0}^N$ or $\sum_{n=1}^N$ or something else?}
\begin{align*}
\tilde{\mathcal{E}}_2&=-\lambda \int_{Q_T}\displaystyle \displaystyle\sum_{i}\sum_{n=1}^{N}\sgn(u_i^{n}-u(s,y)) f(u(s,y))(\nu(c_{i+\frac{1}{2}}^{n-1})-\nu(c_{i-\frac{1}{2}}^{n-1}))
\int\limits_{C_i}\Phi(s,y,t^{n},x)  \d x\d s  \d y\\
&\quad+\lambda \int_{Q_T}\displaystyle \displaystyle\sum_{i}\sum_{n=1}^{N}\sgn(u_i^{n}-u(s,y)) f(u(s,y))(\nu(c_{i+\frac{1}{2}}^{n})-\nu(c_{i-\frac{1}{2}}^{n}))
\int\limits_{C_i}\Phi(s,y,t^{n+1},x)  \d x\d s  \d y\\
&\quad+\lambda \int_{Q_T}\displaystyle \displaystyle\sum_{i}\sgn(u_i^{0}-u(s,y)) f(u(s,y))(\nu(c_{i+\frac{1}{2}}^{0})-\nu(c_{i-\frac{1}{2}}^{0}))
\int\limits_{C_i}\Phi(s,y,t^1,x)  \d x\d s  \d y\\
&\quad-\lambda \int_{Q_T}\displaystyle \displaystyle\sum_{i}\sgn(u_i^{N}-u(s,y)) f(u(s,y))(\nu(c_{i+\frac{1}{2}}^{N})-\nu(c_{i-\frac{1}{2}}^{N}))
\int\limits_{C_i}\Phi(s,y,t^{N+1},x)  \d x\d s  \d y.
\end{align*}
Adding and subtracting the term
\begin{equation*}
-\lambda \int_{Q_T}\displaystyle \displaystyle\sum_{i}\sum_{n=1}^{N}\sgn(u_i^{n}-u(s,y)) f(u(s,y))(\nu(c_{i+\frac{1}{2}}^{n-1})-\nu(c_{i-\frac{1}{2}}^{n-1}))
\int\limits_{C_i}\Phi(s,y,t^{n+1},x)  \d x\d s  \d y,
\end{equation*}
we have
\begin{align*}
\tilde{\mathcal{E}}_2&=-\lambda \int_{Q_T} \displaystyle\displaystyle\sum_{i}\sum_{n=1}^{N}\sgn(u_i^{n}-u(s,y)) f(u(s,y))(\nu(c_{i+\frac{1}{2}}^{n-1})-\nu(c_{i-\frac{1}{2}}^{n-1}))\\
&\qquad \qquad  \times
\int\limits_{C_i}(\Phi(s,y,t^{n},x)-\Phi(s,y,t^{n+1},x))  \d x  \d s \d y\\
&\quad+\lambda \int_{Q_T}\displaystyle \sum_{i}\sum_{n=1}^{N}\sgn(u_i^{n}-u(s,y)) f(u(s,y))\Big[\nu(c_{i+\frac{1}{2}}^{n})-\nu(c_{i-\frac{1}{2}}^{n})-\nu(c_{i+\frac{1}{2}}^{n-1})+\nu(c_{i-\frac{1}{2}}^{n-1})\Big]\\
&\qquad \qquad  \times\int\limits_{C_i}\Phi(s,y,t^{n+1},x) \d x\d s  \d y\\
&\quad+\lambda \int_{Q_T} \displaystyle \displaystyle\sum_{i}\sgn(u_i^{0}-u(s,y)) f(u(s,y))(\nu(c_{i+\frac{1}{2}}^{0})-\nu(c_{i-\frac{1}{2}}^{0}))
\int\limits_{C_i}\Phi(s,y,t^1,x)  \d x \d s \d y\\
&\quad-\lambda \int_{Q_T} \displaystyle \displaystyle\sum_{i}\sgn(u_i^{N}-u(s,y)) f(u(s,y))(\nu(c_{i+\frac{1}{2}}^{N})-\nu(c_{i-\frac{1}{2}}^{N}))
\int\limits_{C_i}\Phi(s,y,t^{n+1},x) \d x \d s \d y\\
&:=\tilde{\mathcal{E}}_{21}+\tilde{\mathcal{E}}_{22} +\tilde{\mathcal{E}}_{23} +\tilde{\mathcal{E}}_{24}.
\end{align*}
Now, let us estimate $\tilde{\mathcal{E}}_{21}$.
\begin{align*}
   \abs{\tilde{\mathcal{E}}_{21}}
   &= \Big|-\lambda \int_{Q_T} \displaystyle \displaystyle\sum_{i}\sum_{n=1}^{N}\sgn(u_i^{n}-u(s,y)) f(u(s,y))(\nu(c_{i+\frac{1}{2}}^{n-1})-\nu(c_{i-\frac{1}{2}}^{n-1}))\\
&\qquad \qquad  \times
\int\limits_{C_i}\big(\Phi(s,y,t^{n},x)-\Phi(s,y,t^{n+1},x)\big) \d x  \d s  \d y\Big|\\
&\le \lambda \norma{f(u)}_{L^{\infty}(Q_T)}\\& \qquad\times \int_{Q_T}\Big( \sum_{i}\sum_{n=1}^{N} 
 \abs{\nu(c_{i+\frac{1}{2}}^{n-1})-\nu(c_{i-\frac{1}{2}}^{n-1})} 
\int\limits_{C_i}\abs{\Phi(s,y,t^{n},x)-\Phi(s,y,t^n+\Delta t,x)} \d x \Big) \d s \d y\\
&\le \D x\, \lambda \norma{f(u)}_{L^{\infty}(Q_T)}\\&\qquad  \times \int_{0}^T \Big( \displaystyle \displaystyle\sum_{i}\sum_{n=1}^{N} 
 \abs{\nu(c_{i+\frac{1}{2}}^{n-1})-\nu(c_{i-\frac{1}{2}}^{n-1})}
\abs{\omega_{\epsilon_0}(s-t^n)-\omega_{\epsilon_0}(s-t^n-\Delta t)}  \Big) \d s \\
% &\le &
% \mathcal{C}_3\lambda \norma{f(u)}_{L^{\infty}(Q_T)} \sum_{n=1}^{N}  \int_{Q_T} \Big( \displaystyle \displaystyle
% \int\limits_{C_i}\abs{\omega_{\epsilon}(y-x)}\abs{\omega_{\epsilon_0}(s-t^n)-\omega_{\epsilon_0}(s-t^n-\Delta t)}  \Big) \d x\d s \d y \\
&\le  \D t \, \mathcal{C}_3 \norma{f(u)}_{L^{\infty}(Q_T)} \sum_{n=1}^{N} \int_{0}^T \Big( \abs{\omega_{\epsilon_0}(s-t^n)-\omega_{\epsilon_0}(s-t^n-\Delta t)} \Big) \d s \\
% \\ &\leq& 
% \D x \lambda \norma{f(u)}_{L^{\infty}(Q_T)} \sum_{n=1}^{N} \mathcal{C}_3 \norma{u^{n-1}}_{L^1(\R)}\int_{0}^T \Big( \displaystyle \displaystyle
% \abs{(\omega_{\epsilon_0}(t^n-s)-\omega_{\epsilon_0}(t^n+\Delta t-s))} \Big) \d s  \\
&\le \Delta t \, \mathcal{C}_3\norma{f(u)}_{L^{\infty}(Q_T)} \sum_{n=1}^{N}  \abs{\omega_{\epsilon_0}}_{BV(\R)}\Delta t\\
% &\le &\mathcal{C}_4\lambda \norma{f(u)}_{L^{\infty}(Q_T)} 
%  \frac{\Delta x}{\epsilon_0}\\
&=\mathcal{C}_4\frac{\Delta t}{\epsilon_0}.
\end{align*}
Now, let us estimate $\tilde{\mathcal{E}}_{22}$.
\begin{align*}
    \abs{\tilde{\mathcal{E}}_{22}}&=\Big|\lambda \int_{Q_T}  \displaystyle \displaystyle\sum_{i}\sum_{n=1}^{N}\sgn(u_i^{n}-u(s,y)) f(u(s,y))\big[\nu(c_{i+\frac{1}{2}}^{n})-\nu(c_{i-\frac{1}{2}}^{n})-\nu(c_{i+\frac{1}{2}}^{n-1})+\nu(c_{i-\frac{1}{2}}^{n-1})\big]\\
& \qquad  \times\int\limits_{C_i}\Phi(s,y,t^{n+1},x) \d x  \d s \d y\Big|\\
&\le \lambda \abs{\nu}_{\lip(\R)}\norma{f(u)}_{L^{\infty}(Q_T)} \int_{Q_T} \sum_{i}\sum_{n=1}^{N}\big|c_{i+\frac{1}{2}}^{n}-c_{i-\frac{1}{2}}^{n}-c_{i+\frac{1}{2}}^{n-1}+c_{i-\frac{1}{2}}^{n-1}\big|\\&\qquad \times \int\limits_{C_i}|\Phi(s,y,t^{n+1},x)| \d x  \d s \d y\\
&= \Delta t\abs{\nu}_{\lip(\R)}\norma{f(u)}_{L^{\infty}(Q_T)}  \sum_{i}\sum_{n=1}^{N}\big|c_{i+\frac{1}{2}}^{n}-c_{i-\frac{1}{2}}^{n}-c_{i+\frac{1}{2}}^{n-1}+c_{i-\frac{1}{2}}^{n-1}\big|.
\end{align*}
Recall, cf.\ \eqref{eq:conv}, that
\[
c_{i+1/2}^n=\Delta x \sum_{p}\beta(u_{p+1/2}^n)\mu_{i+1/2-p},
\]
which implies
\begin{align*}  
&{\sum_{i}\big|(c_{i+1/2}^n-c_{i-1/2}^n)-(c_{i+1/2}^{n-1}-c_{i-1/2}^{n-1})\big|}\\
&\quad\le \Delta x \abs{\beta}_{\lip(\R)}\sum_{i}\Big|\big(\sum_{p}u_{p+1/2}^n\mu_{i+1/2-p}-\sum_{p}u_{p+1/2}^n\mu_{i-1/2-p}\big)\\
&\qquad -\big(\sum_{p}u_{p+1/2}^{n-1}\mu_{i+1/2-p}-\sum_{p}u_{p+1/2}^{n-1}\mu_{i-1/2-p}\big)\Big|\\
&\quad =\Delta x \abs{\beta}_{\lip(\R)}\sum_{i}\big|\sum_{p}u_{p+1/2}^n(\mu_{i+1/2-p}-\mu_{i-1/2-p})\\
&\qquad -\sum_{p}u_{p+1/2}^{n-1}(\mu_{i+1/2-p}-\mu_{i-1/2-p})\big|\\
% &=&\sum_{i}\Delta x\Big|\sum_{p}(u_{p+1/2}^n-u_{p+1/2}^{n-1})(\mu_{i+1/2-p}-\mu_{i-1/2-p})\Big|\\
&\quad \le\Delta x\abs{\beta}_{\lip(\R)}{\sum_{i,p}\big|u_{p+1/2}^n-u_{p+1/2}^{n-1}\big|\big|\mu_{i+1/2-p}-\mu_{i-1/2-p}\big|}\\
&\quad \le  \Delta x\abs{\beta}_{\lip(\R)}|\mu|_{BV(\R)}\sum_{p}\big|u_{p+1/2}^n-u_{p+1/2}^{n-1}\big|\\
&\quad \le\mathcal{C}_5\Delta t,
\end{align*}
by applying \eqref{apx:time}.
Therefore,
\begin{eqnarray*}
   \tilde{\mathcal{E}}_{22}
   &\le &\mathcal{C}_6\Delta t. \end{eqnarray*}
Finally, estimates on the remaining boundary terms $\tilde{\mathcal{E}}_{23}$ and $\tilde{\mathcal{E}}_{24}$, easily follow  from \eqref{cc}.   Specifically,
\begin{align*}
      \tilde{\mathcal{E}}_{23}&=\lambda \int_{Q_T}\displaystyle \displaystyle\sum_{i}\sgn(u_i^{0}-u(s,y)) f(u(s,y))(\nu(c_{i+\frac{1}{2}}^{0})-\nu(c_{i-\frac{1}{2}}^{0}))
\int\limits_{C_i}\Phi(s,y,t^1,x)  \d x\d s  \d y \\
&\le \mathcal{C}_7\Delta t\, \norma{f(u)}_{L^{\infty}(Q_T)}. 
\end{align*}
Similarly,
\[\tilde{\mathcal{E}}_{24}
\le \mathcal{C}_7\Delta t \,\norma{f(u)}_{L^{\infty}(Q_T)}.
\]
Substituting the assertions of \ref{C1} and \ref{C2} in \eqref{est:5} we get
\[
-\Lambda_{\epsilon,\epsilon_0} (u^{\Delta},u)\leq \mathcal{C}\left(\frac{\D x}{\epsilon}+\frac{\D t}{\epsilon_0}\right).
\]
\end{proof}
Now, we state and prove the main result of this paper.
\begin{theorem}\label{CR} [Rate of Convergence]
Let $u$ be the entropy solution of \eqref{IVP:eq}--\eqref{IVP:data} and $u^{\D}$ be the numerical solution given by \eqref{scheme2}.
Then we have the following convergence rate:
 \begin{equation*}
 \norma{u^{\D}(T,\dott)-u(T,\dott)}_{L^1(\R)} = \mathcal{O}(\sqrt{\Delta t}).
 \end{equation*}
\end{theorem}
\begin{proof}
The CFL condition implies that $\D x =\mathcal{O} (\D t).$
Furthermore, the initial approximation \eqref{apx:data} implies $\norma{u_0^{\D}-u_0}_{L^1(\R)}=\mathcal{O}(\D t)$.
Now, the desired error estimate follows from Lemma~\ref{lemma:est} and Lemma~\ref{lemma:kuz} setting  $\epsilon=\epsilon_0=\sqrt{\Delta t},$ as $\gamma(u^{\D},\sqrt{\D t}) = \mathcal{O}(\sqrt{\D t}).$
\end{proof}
% \section{Nonlocal conservation laws in several space dimension}
% In this section, we briefly discuss the rate of convergence of dimension splitting method for nonlocal conservation laws and its rate of convergence. For the sake of simplicity we restrict ourselves to $d=2$ and higher dimensions follow analogously. We consider the IVP
% \begin{eqnarray}
%  \label{IVP:eq}
%   \partial_t u
%   +
%   \partial_x (f(u) \nu(\mu*\beta(u)))+\partial_y (g(u) \nu(\mu*\beta(u)))&=& 0\, \quad \quad \quad \text{for}\,\,\,(t,x) \in(0,\infty)\times \R,
%   \\ \label{IVP:data}
%   u(0,x)&=&u_0(x)  \quad \,\, \text{for}\,\,\,x \in \R.
% \end{eqnarray}  
% At this point we would like to remark that the doubling of the variable argumets can be carried out as in and hence the uniqueness of the entropy solution follows.

\section{Extension to Multi Dimensions}\label{exis2}
% The above analysis can be extended to several space dimensions as well, i.e., for the following PDE:
% \begin{equation}
%   \label{eq:U}
%   \partial_t u + \mathrm{div}_{\mathbf{x}} \,F (u,\mu * u) = 0
% \end{equation}
% with a nonlocal flow of the type
% \begin{displaymath}
%   \begin{array}{lllllllllll}
%     F & \colon &
%     \R^+ &  \times & \R & \times & \R
%     & \to & \R^{d}
%     \\
%     & & t && u && \mu*u & \to &  F (u,\mu*u).
%   \end{array}
% \end{displaymath} 
% In particular, aiming at a formal
% simplification, 
We consider the case of two space
dimensions and denote the space variables by $(x,y) \in
\R^2$, and consider the following PDE:
\begin{equation}
  \label{eq:1}
  \partial_t u
  +
  \partial_x (f^1(u)\nu^1(\beta^1(u)*\mu^1))
  +
  \partial_y (f^2(u)\nu^2(\beta^2(u)*\mu^2))
  =
  0.
\end{equation}
Further, for numerical scheme, fix a rectangular grid with sizes $\Delta x$ and $\Delta y$ in
$\R^2$ and choose a time step $\Delta t$. For later use, we also
introduce the usual notation
\begin{equation*}
  (t^n, x_i,y_j) = (n\Delta t,i\Delta x, j \Delta y), \quad
  n\in\N,\, i, j\in \Z, \quad
    \lambda_x  = \frac{\Delta t}{\Delta x},  \,
    \lambda_y  =  \frac{\Delta t}{\Delta y}. 
\end{equation*}
Throughout,  we fix initial data $u_0\in (L^{\infty}\cap \operatorname{BV}) (\R^2;
\R)$ and introduce
\[
  u^{0}_{ij}
  =
  \frac{1}{\Delta x \, \Delta y}
  \int_{x_{i-1/2}}^{x_{i+1/2}} \int_{y_{j-1/2}}^{y_{j+1/2}}
  u_0(x,y) \, \d{y}\, \d{x} 
  \quad \mbox{ for } i,j \in \Z.
\]
We define a piecewise constant approximate solution $u^{\Delta}$ 
by
\[
  u^{\Delta}(t,x,y) =  u^n_{ij}\chi_{[t^n, t^{n+1}) \times[x_{i-1/2},x_{i+1/2})\times [y_{j-1/2},y_{j+1/2})}(t,x,y),  \quad
     n\in\N, \,  i,j \in  \Z,
 \]
 where $\chi_A$ denotes the indicator function of a set $A$, 
through the following marching formula based on dimensional splitting,
(see~\cite[Sec.~3]{CrandallMajda1980Monotone} and \cite[Sec.~5]{holden2010splitting} for details):
\begin{align}
  \label{eq:2} 
    u^{n+1/2}_{ij}
    & = 
    u^n_{ij}
    - 
    \lambda_x \bigl[
   \mathcal{F}^{x,n}_{i+1/2,j} (u^n_{ij},u^n_{i+1,j})
     - 
   \mathcal{F}^{x,n}_{i-1/2,j} (u^n_{i-1,j},u^n_{ij})
     \bigr],
    \\[6pt]  \nn
    u^{n+1}_{ij}
    & = 
    u^{n+1/2}_{ij}
     - 
    \lambda_y \bigl[
   \mathcal{F}^{y,n}_{i,j+1/2} (u^{n+1/2}_{ij},u^{n+1/2}_{i,j+1})
    - 
   \mathcal{F}^{y,n}_{i,j-1/2} (u^{n+1/2}_{i,j-1},u^{n+1/2}_{ij})
    \bigr],
\end{align}
where $\mathcal{F}^{x,n}_{i+1/2,j}$ and  $\mathcal{F}^{y,n}_{i,j+1/2}$ denote the numerical approximations of the fluxes $f^1(u) \nu^1(\mu^1*\beta^1(u))$ and $f^2(u) \nu^2(\mu^2*\beta^2(u))$ at the interfaces $(x_{i+1/2},y_j)$ and $(x_{i},y_{j+1/2})$, respectively, for $i,j\in \Z$.
The convolution terms are computed through quadrature formula, i.e.,
\begin{equation}
  \label{eq:AB}
    c^{x,n}_{\strut i+1/2, j}
    =  
    \Delta x \, \Delta y
      \sum_{\strut l,p \in \interi} 
      \mu^{1}_{\strut i+1/2-l, j-p}  \beta^1(u^n_{\strut l+1/2, p}), \quad
    c^{y,n}_{\strut i, j+1/2}
    = 
    \Delta x \, \Delta y
      \sum_{\strut l,p \in \interi}
      \mu^2_{\strut i+1/2-l, j-p}   \beta^2(u^n_{\strut l+1/2, p}) ,
\end{equation}
where, $u^n_{l+1/2,p}$ is any  convex combination of
$u^n_{l,p}$ and $u^n_{l+1,p}$, with
% \begin{equation*}
  % \label{eq:etatheta}
$\mu^1_{i+1/2,j} = \mu^1 (x_{i+1/2}, y_j)$ and $
\mu^2_{i+1/2,j} = \mu^2 (x_{i+1/2}, y_j) \,.$
% \end{equation*}
Throughout, we require that $\Delta t$ is chosen in order to satisfy
the CFL conditions
\begin{equation}
   \label{CFL_LF1}
 \lambda_x\le \frac{\min(1, 4-6\theta_x,6\theta_x )}{1+6\abs{f^1}_{\lip(\R)}\norma{\nu^1}_{L^\infty(\R)}},
  \quad \quad 
  \lambda_y\le \frac{\min(1, 4-6\theta_y,6\theta_y )}{1+6\abs{f^2}_{\lip(\R)}\norma{\nu^2}_{L^\infty(\R)}}, \quad \quad 
  \theta_x,\theta_y \in \left(0,\frac{2}{3} \right),
\end{equation}
and 
\begin{equation*}
   \lambda _x\abs{f^1}_{\lip(\R)}\norma{\nu^1}_{L^\infty(\R)}\leq \frac12, \quad
   \lambda _y\abs{f^2}_{\lip(\R)}\norma{\nu^2}_{L^\infty(\R)}\leq \frac12,
\end{equation*}
   with numerical fluxes $\mathcal{F}^x$ or $\mathcal{F}^y$ chosen as 
   Lax--Friedrichs flux and Godunov flux, respectively. Extension to other monotone fluxes and for higher dimensions is similar. The numerical scheme can now be shown to converge to entropy solution,  see, for example, \cite{aggarwal2015nonlocal}. 
The Kuznetsov Lemma and the theorem on error estimate presented in Section \ref{EE} can now be extended to several space dimensions using dimension splitting arguments (see~\cite[Sec.~4.3]{holden2015front})  with appropriate modifications throughout the proof.
% \end{remark}

\section{Numerical Experiments}
\label{NR}
% \subsection{Scalar Paper}
% \section{Numerical Experiments}
% \label{NR}
We now present some numerical experiments to illustrate the theory presented in the previous section. We show the results for the Lax--Friedrichs scheme. The results obtained by  Godunov scheme  are similar, and are not shown here. Throughout the section, $\theta=\theta_x=\theta_y$ is chosen to be $0.3333$, and $\lambda$ and $\lambda_x=\lambda_y$ are chosen to be $0.1286$ and $0.2857$, respectively, so as to satisfy the CFL condition \eqref{CFL_LF} and \eqref{CFL_LF1}, respectively, in one and two dimensions, for any grid size $\Delta x$ or $\Delta y$ used in this section. 

\subsection{One Dimension}
We employ the nonlocal version of the standard LWR model \eqref{LWR}, i.e., IVP \eqref{IVP:eq}--\eqref{IVP:data}, with 
\[
\mu(x)=\frac{3}{\eta^3}(\eta-x)^2\mathbbm{1}_{(0,\eta)}(x), \quad \eta =0.0625,
\]
$f(u)=u, \beta(r)=r$ and $ \nu(r)=1-r$. This PDE fits the hypothesis of the article.
Further, the domain of integration is chosen to be the interval $[-1.5, 1.5]$ with $t\in[0, 0.5]$, and 
% with $$f^k(u)=u, \beta^k(x,u)=s^k(x)u, \mu(x)=\frac{3}{\epsilon^3}(\epsilon-x)^2\mathbbm{1}_{(0, \epsilon)}(x)$$ where $s^k(x)$ is chosen appropriately so that the flux $\boldsymbol{F}$ satisfies (H0)-(H2).
\begin{eqnarray}
    \label{eq:ex1} u_0(x)=0.25\mathbbm{1}_{(-0.9,0.3)}(x)+0.5\mathbbm{1}_{(0.1,0.3)}(x).
\end{eqnarray}
% \end{equation}

Figure \ref{fig:ex1} displays the numerical approximations of \eqref{IVP:eq}, \eqref{eq:ex1} generated by the numerical scheme \eqref{scheme2}, 
\begin{figure}[ht!]
\includegraphics[width=\textwidth,keepaspectratio]{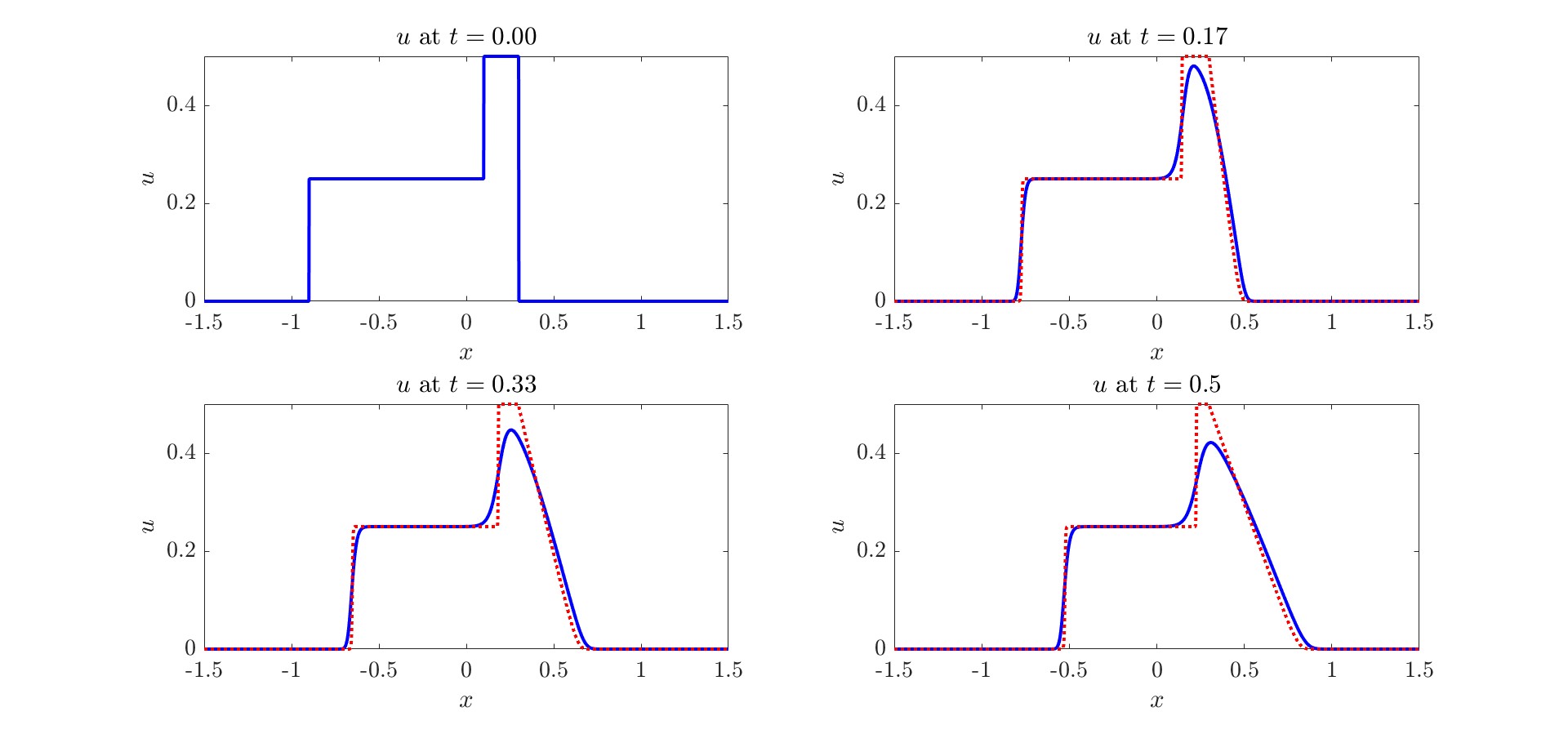}  
\caption{Solution to the nonlocal conservation law~\eqref{IVP:eq}, \eqref{eq:ex1} ({\color{blue}\full}).  Reference solution to the local conservation law ~\eqref{LWR}, \eqref{eq:ex1}  ({\color{red}\dotted}) on the domain $[-1.5, \, 1.5]$ at times $t =
    0.00,\; 0.017,\;0.33, \: 0.5$, and  $\Delta x =0.0015625$.}
  \label{fig:ex21}
\end{figure} 
using space mesh
    $\Delta x =0.0015625$. The numerical solution of the corresponding local version 
    \eqref{LWR}, \eqref{eq:ex1}, obtained using standard Godunov scheme, has also been plotted at the respective times alongside the nonlocal solutions, as a reference to indicate the difference between the solutions of these two PDEs. Figure \ref{fig:ex21} shows the the numerical scheme is able to capture both shocks and rarefactions well, and that the density does not cross the maximal density 1.

To compute the convergence rate of the scheme~\eqref{scheme2}, we obtain numerical approximations to ~\eqref{IVP:eq}, \eqref{eq:ex1}  with
decreasing grid sizes $\Delta x$, starting with $\Delta x =0.00625$.  The convergence rate $\alpha$
is then calculated at time $T=0.5$ by computing the $L^1$ distance
between the numerical solutions $u_{\Delta x}(T,\dott)$ and $\displaystyle u_{\Delta x/2}(T,\dott)$ obtained
for the grid size $\Delta x$ and $\frac{1}{2}\Delta x$, for each grid size $\Delta x$. The results recorded in Table \ref{Table81} and Figure \ref{fig:my_label21} show
that the observed convergence rates lie strictly between $0.5$ and $1$.  
\begin{figure}[ht!]
\centering
\includegraphics[width=\textwidth,keepaspectratio]{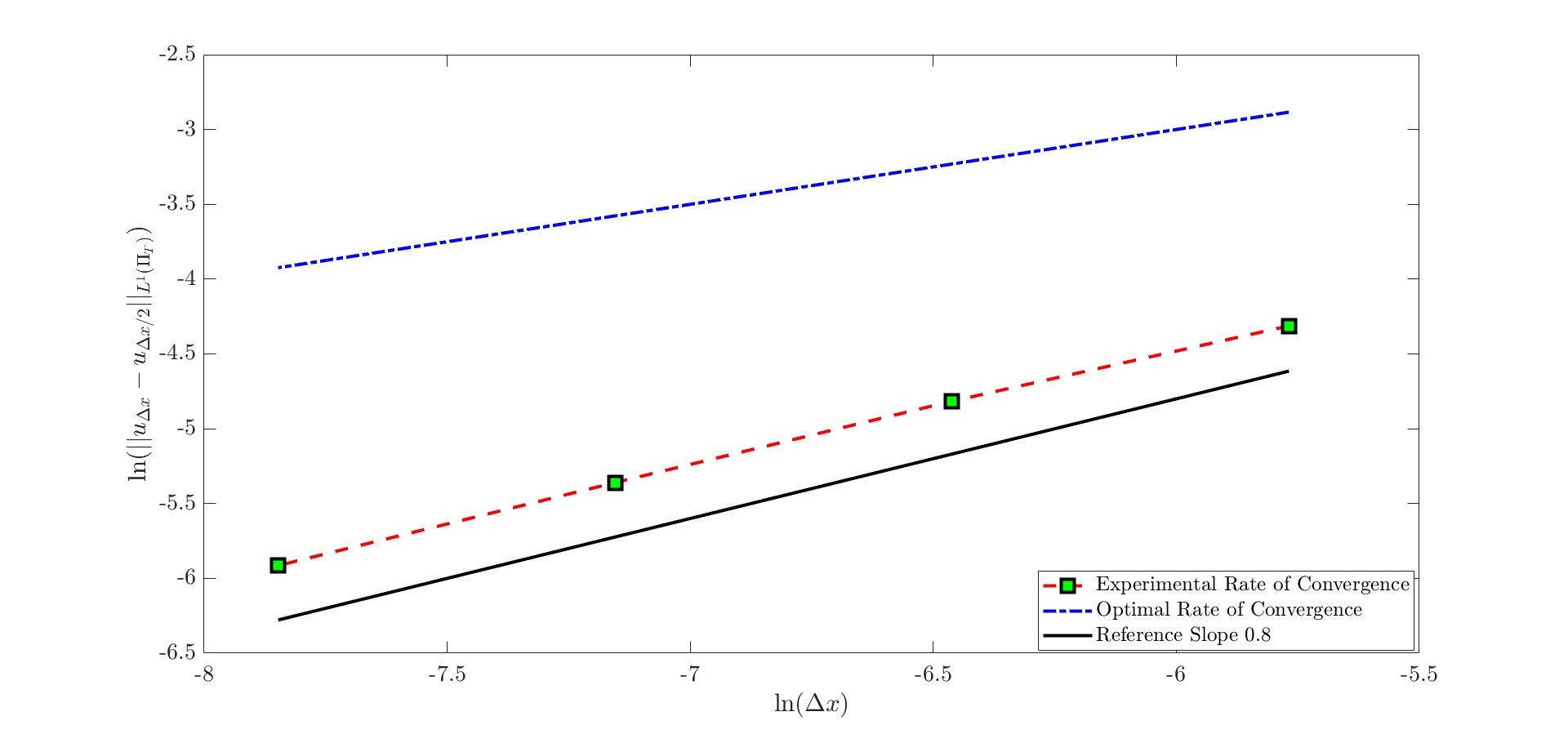}
    \caption{Convergence rate $\alpha$ for the numerical scheme~\eqref{scheme2}  for the approximate solutions
  to the problem~\eqref{IVP:eq}, \eqref{eq:ex1} on the domain $[-1.5, \, 1.5]$ at time $T=0.5.$}
   \label{fig:my_label21}
\end{figure}
\begin{table}[ht!]{
 \centerline{
   \begin{tabular}{|c|c|c|c|c|c|c|c|c|c|}\hline
     \multicolumn{1}{|c|}{ $\displaystyle{\D x}$}&\multicolumn{1}{|c|}{$\displaystyle\norma{u_{\Delta x}(T,\dott)-u_{\Delta x/2}(T,\dott)}_{L^1(\R)}$}\vline & \multicolumn{1}{|c|}{$\displaystyle\log_{2}\frac{\norma{u_{\Delta x}(T,\dott)-u_{\Delta x/2}(T,\dott)}_{L^1(\R)}}{\norma{u_{\Delta x/2}(T,\dott)-u_{\Delta x/4}(T,\dott)}_{L^1(\R)}}$}\vline\\
     % \hline
     % \backslashbox{$\Delta x$}{Initial Data}
     % &&\tabularnewline
     \hline
     $0.00625$&$0.0034$&$0.7262$\tabularnewline
     \hline
     $0.003125$&$0.0081$&$0.7853$\tabularnewline
     \hline
     $0.0015625$&$0.0047$&$0.7997$\tabularnewline
     \hline
$0.00078125$&$0.0027$&\tabularnewline
     \hline
    \end{tabular}}}
\caption{Convergence rate $\alpha$ for the numerical scheme~\eqref{scheme2} on the domain $[-1.5, \, 1.5]$ at time $T=0.5$ for the approximate solutions
  to the problem~\eqref{IVP:eq}, \eqref{eq:ex1}.}\label{Table81}
\end{table}
The present numerical integration resonates well with the theoretical convergence rate obtained in Theorem \ref{CR} in this article.

Figure \ref{fig:ex211} illustrates the nonlocal to local limit, see \cite{bressan2020traffic,colombo2019singular,keimer2019approximation} and references therein, namely that the entropy solutions of the nonlocal conservation laws converge to the entropy solution of the corresponding local conservation law as the radius of the kernel goes to zero. This numerical example indicates that also the present model has singular solutions. 
\begin{figure}[ht!]
\includegraphics[width=\textwidth,keepaspectratio]{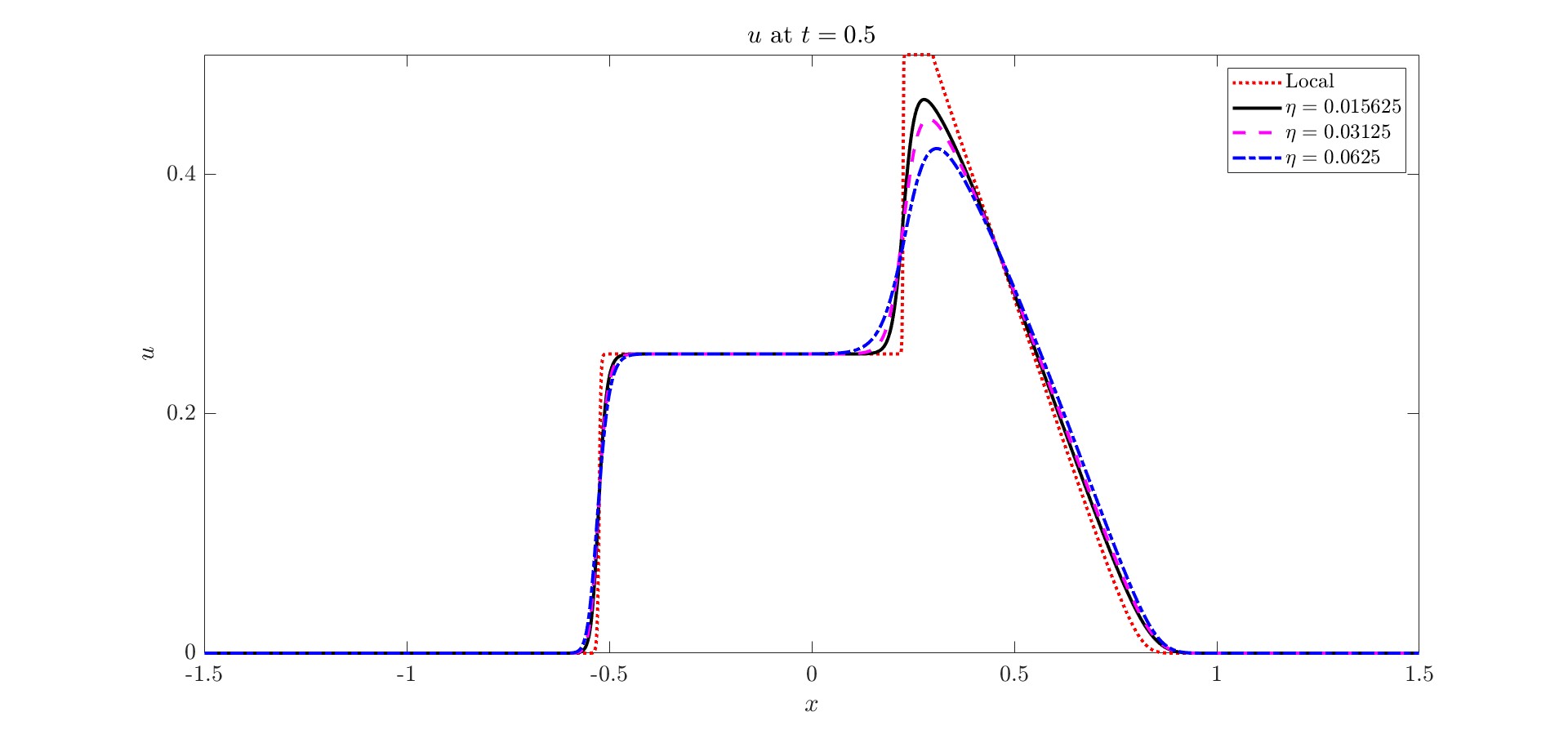}  \caption{Domain $[-1.5, \, 1.5],  T=0.5, \Delta x =0.0015625$: Solution to the local conservation law ~\eqref{LWR}, \eqref{eq:ex1} ({\color{red}\dotted}); Solution to the nonlocal conservation law~\eqref{IVP:eq}, \eqref{eq:ex1} with decreasing convolution radii $\eta=0.015625,0.03125,0.0625$.}
  \label{fig:ex211}
\end{figure}
\subsection{Two Dimensions}
To illustrate our results in two dimensions, we employ the model introduced
in~\cite{aggarwal2015nonlocal}, modeling crowd dynamics in two dimensions, which fits in the framework
of the article. Assume that a group of pedestrians 
%{\color{blue} along a corridor} 
in a square room $[-4,4]^2$, can be described
through the density $u = u (t,x,y)$ that satisfies the nonlocal
conservation law
\begin{equation}
  \label{eq:5}
  \partial_t u
  +
  \nabla \cdot \left(
    u \, (1 - u )  \, (1 - u * \mu) 
  \right)
  =
  0,
\end{equation}
where the smooth, non-negative and compactly supported function
$\mu$ models the way in which each individual averages the density
around her/his position to adjust her/his speed.
\begin{figure}[ht!]
  \centering
\includegraphics[width=0.45\textwidth,keepaspectratio]{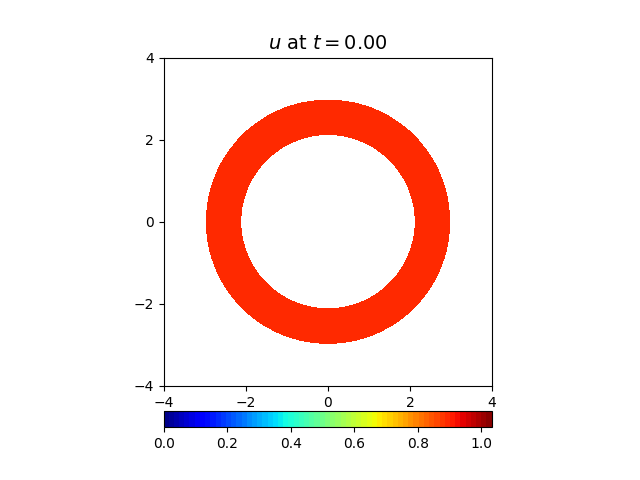}
\includegraphics[width=0.45\textwidth,keepaspectratio]{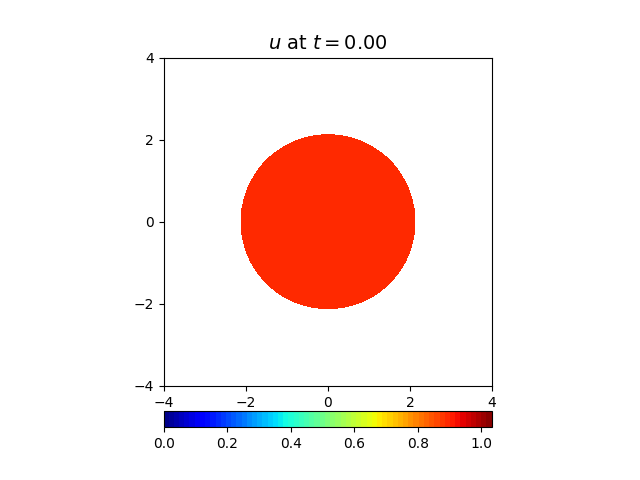}\\
\includegraphics[width=0.45\textwidth,keepaspectratio]{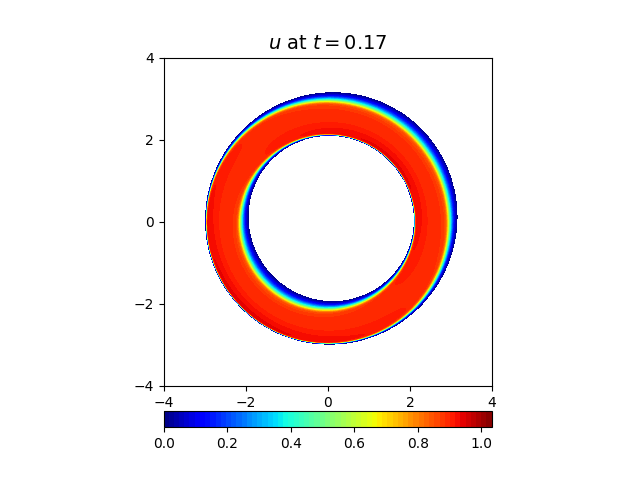}
\includegraphics[width=0.45\textwidth,keepaspectratio]{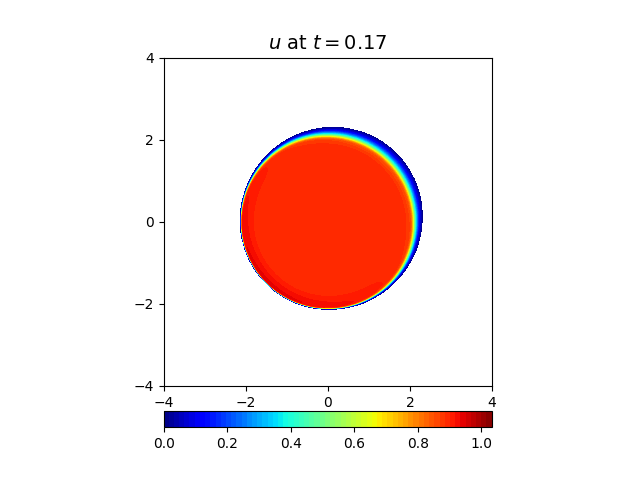}\\
\includegraphics[width=0.45\textwidth,keepaspectratio]{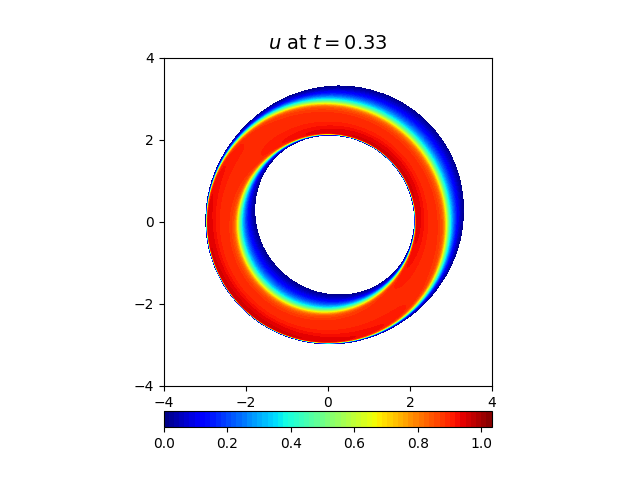}
\includegraphics[width=0.45\textwidth,keepaspectratio]{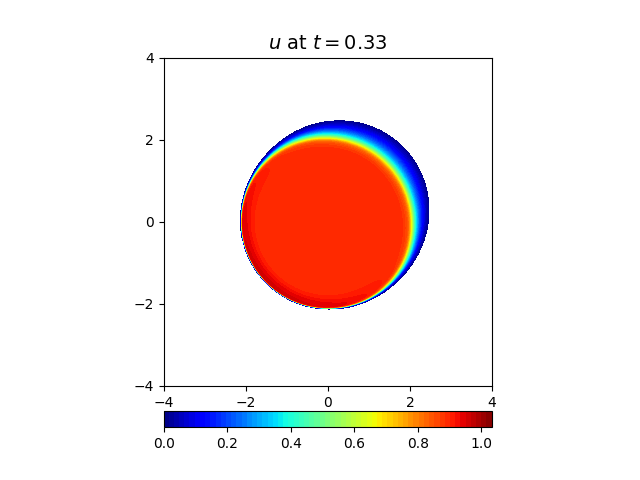}\\
\includegraphics[width=0.45\textwidth,keepaspectratio]{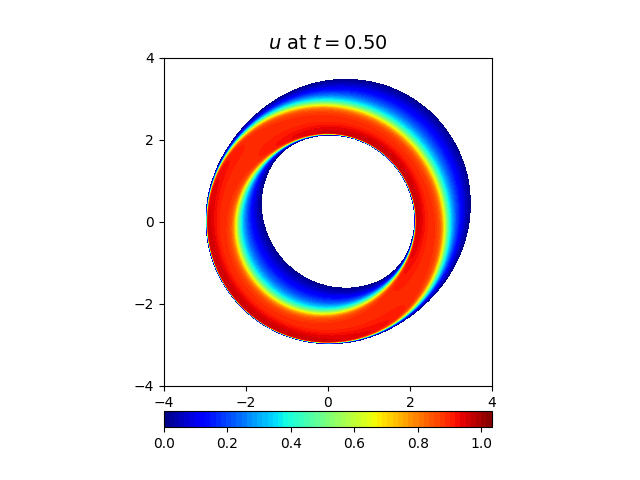}
\includegraphics[width=0.45\textwidth,keepaspectratio]{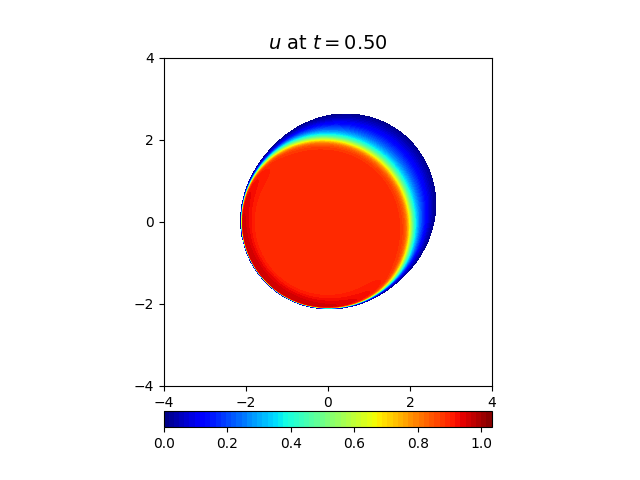}\\
  \caption{Solution to~\eqref{eq:5}, \eqref{eq:ex1:ID} (Left)  and  ~\eqref{eq:5}, \eqref{eq:ex2:ID} (Right) at times $t =
    0.00,\; 0.017,\;0.33, \: 0.5$ with space mesh
    $\Delta x = \Delta y = 0.00625$.}
  \label{fig:ex1}
\end{figure}
% \begin{lemma}
%   \label{lem:ex1}
%   System~\eqref{eq:5} fits into~\eqref{IVP:eq} setting
%   \begin{equation}
%     \label{eq:ex1:table}
%     \begin{array}{@{}r@{\,}c@{\,}l}
%       N & = & 1\,,
%       \\
%       m & = & 1\,,
%     \end{array}
%     \quad
%     \begin{array}{r@{\,}c@{\,}l}
%       f (t,x,y,U,A)
%       & = &
%       U \, (1-U)\, (1-A) \, v^1 (x,y)
%       \\
%       g (t,x,y,U,B)
%       & = &
%       U \,  (1-U)\, (1-B) \, v^2 (x,y)
%     \end{array}
%     \quad \mbox{ and } \quad
%     \begin{array}{r@{\,}c@{\,}l@{}}
%       \eta & = & \mu
%       \\
%       \theta & = & \mu \,.
%     \end{array}
%   \end{equation}
%   Moreover, if $\mu \in (\C2 \cap \W{2}{\infty}) (\R^2; \R)$,
%   then~Theorem~\ref{thm:main} applies to any initial datum in $(\L1
%   \cap L^{\infty}\cap \operatorname{BV}) (\R^2; [0,1])$.
% \end{lemma}

% \noindent The proof in Section~\ref{sec:TD} essentially relies on the
% invariance of the interval $[0,1]$ for $u$.
We choose:
\begin{equation}
  \label{eq:6}
\tilde \mu (x,y)
   = 
  (0.16 - x^2 - y^2)^3 \;
  \chi_{\strut \{(x,y) \colon x^2 + y^2 \leq 0.16\}} (x,y),
  \quad
  \mu (x,y)
  =
  \big(\iint_{\R^2} \tilde \mu \d x\d y\big)^{-1}\tilde \mu (x,y)\d{x} \d{y},
\end{equation}
so that $\iint_{\R^2} \mu (x,y) \d{x} \d{y} =1$.  

 As initial
data, we consider two cases:
\begin{enumerate}
    \item Annular initial data, where the crowd is concentrated in an annulus:
    \begin{equation}
  \label{eq:ex1:ID}
  u_0 (x,y)
  =
  \chi_{\strut \{(x,y) \colon 4 \le x^2 + y^2\le 9\}} (x,y). 
\end{equation}
\item Circular initial data, where the crowd is concentrated in a circle: \begin{equation}
  \label{eq:ex2:ID}
  u_0 (x,y)
  =
  \chi_{\strut \{(x,y) \colon  x^2 + y^2\le 4\}} (x,y).
\end{equation}
\end{enumerate}

  The system~\eqref{eq:5} fits into the setting of~\eqref{eq:2} with
  \begin{equation*}
    \label{eq:ex1:table}
      f^k(u)=u(1-u), \, \nu^k(u)=1,\, \beta^k(u)=1-u, \,
       \mu^k  = \mu, \quad k=1,2.
  \end{equation*}
The numerical integrations of~\eqref{eq:5}, \eqref{eq:ex1:ID} and~\eqref{eq:5}, \eqref{eq:ex2:ID} are obtained by the
the algorithm described in Section~\ref{exis2} and are shown in
Figure~\ref{fig:ex1}. The figures depict that the density does not cross the maximal density 1 and the numerical simulations are able to capture the physical properties well.

To compute the convergence rate of the scheme~\eqref{eq:2} with 
Lax--Friedrichs flux, we apply
the algorithm to problem~\eqref{eq:5}, \eqref{eq:ex1:ID} and ~\eqref{eq:5}, \eqref{eq:ex2:ID} on the domain
$[-4, \, 4]^2$ on the time interval $[0,0.5]$ with
different grid sizes with $\lambda_x = \lambda_y = 0.2857$. The convergence rate $\alpha$
is then calculated at time $T=0.5$ by computing the $L^1$ distance
between the numerical solutions $u_{\Delta x}(T,\dott)$ and $\displaystyle u_{\frac{\Delta x}{2}}(T,\dott)$ obtained
for the grid size $\Delta x$ and $\frac{\Delta x}{2}$, for each grid size $\Delta x$. The results recorded in Table \ref{Table8} and Figure \ref{fig:my_label1} show
that the observed convergence rates lie strictly between $0.5$ and $1$. 
\begin{figure}[ht!]
\centering\includegraphics[width=.9\textwidth,keepaspectratio]{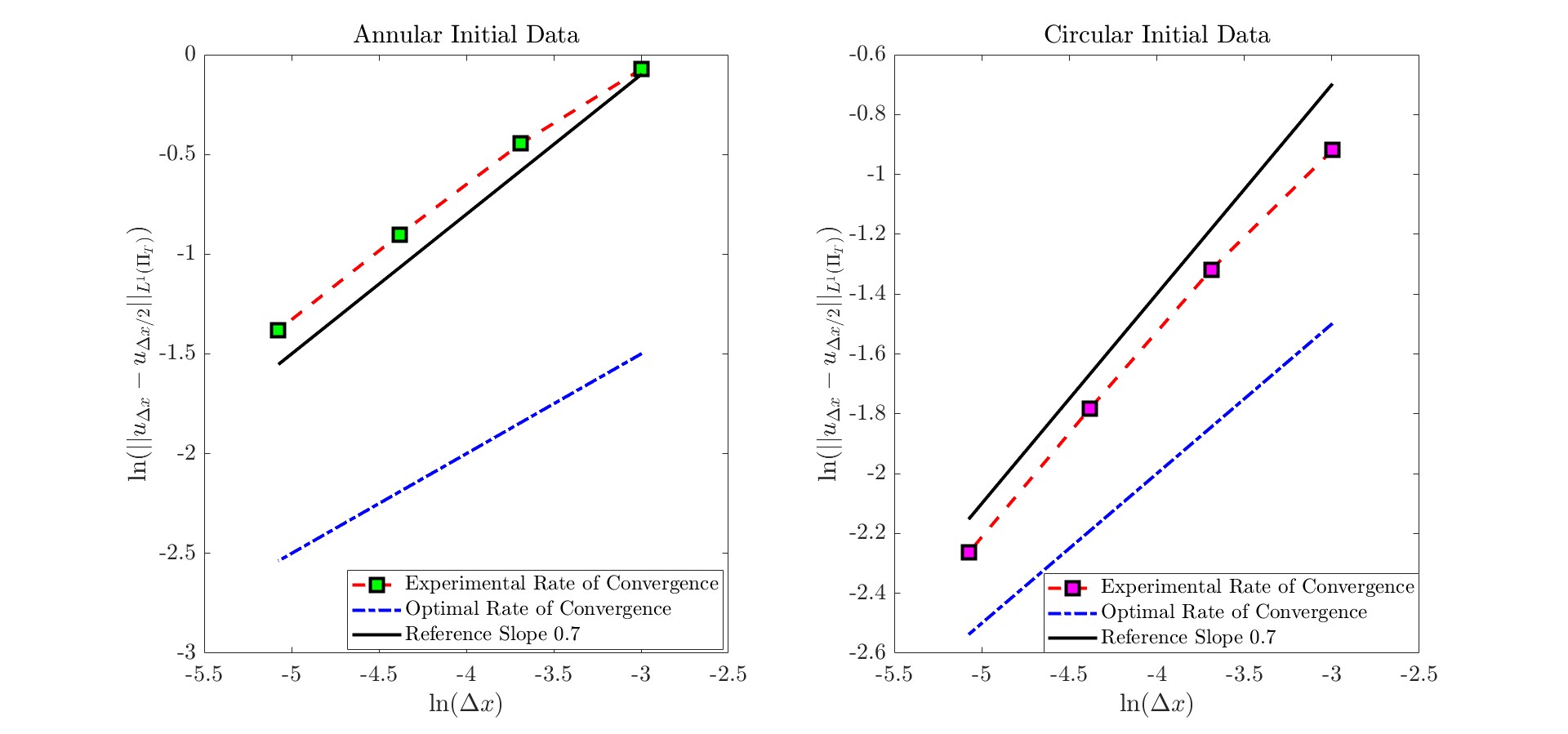}
    \caption{Convergence rate $\alpha$ for the numerical scheme~\eqref{scheme2} on the domain $[-4, \, 4]^2$ at time $T=0.5$ for the approximate solutions
  to the problem~\eqref{eq:5}, \eqref{eq:ex1:ID} and~\eqref{eq:5}, \eqref{eq:ex2:ID}.}
   \label{fig:my_label1}
\end{figure}
\begin{table}[ht!]{
 \centerline{
   \begin{tabular}{|c|c|c|c|c|c|c|c|c|c|}\hline
     \multicolumn{1}{|c|}{}&\multicolumn{2}{|c|}{$\displaystyle\norma{u_{\Delta x}(T,\dott)-u_{\frac{\Delta x}{2}}(T,\dott)}_{L^1(\R)}$}\vline & \multicolumn{2}{|c|}{$\displaystyle\log_{2}\frac{\norma{u_{\Delta x}(T,\dott)-u_{\frac{\Delta x}{2}}(T,\dott)}_{L^1(\R)}}{\norma{u_{\frac{\Delta x}{2}}(T,\dott)-u_{\frac{\Delta x}{4}}(T,\dott)}_{L^1(\R)}}$}\vline\\
     \hline
     \backslashbox{$\Delta x$}{Initial Data}
     & Annular & Circular& Annular & Circular\tabularnewline
     \hline
     $0.05$&  $0.9314$&$0.3989$&$0.5406$&$0.5425$\tabularnewline
     \hline
     $0.025$&$0.6403$ &$0.2677$&$0.6580$&$0.6704$ \tabularnewline
     \hline
     $0.0125$&$0.4057$ &$0.1682$&$0.6901$&$0.6954$\tabularnewline
     \hline
     $0.00625$&$0.2515$ &$0.1039$&&\tabularnewline
     \hline
\end{tabular}}}
\caption{Convergence rate $\alpha$ for the numerical scheme~\eqref{scheme2} on the domain $[-4, \, 4]^2$ at time $T=0.5$ for the approximate solutions
  to the problem~\eqref{eq:5}, \eqref{eq:ex1:ID} and~\eqref{eq:5}, \eqref{eq:ex2:ID}.}\label{Table8}
\end{table}
The present numerical integration resonate well with theoretical convergence rate obtained in Theorem \ref{CR} in this article. 
% The results obtained by  Godunov scheme  are similar, and are not shown here.

\section*{Conclusions}\label{fu}
In this article, we have established the convergence rate estimates for scalar nonlocal nonlinear conservation laws, modeling traffic and crowd dynamics, with no additional assumptions on monotonicity/linearity of the kernel or the flux. The rate is shown to be 1/2 which is consistent with its local counterparts. It is interesting to see that the results are independent of the radius of the convolution matrix and monotonicity of the kernel. For $\nu(u)=\beta(u)=1,$ the nonlocal conservation laws boil down to the local conservation law and hence the convergence rate $1/2$ is optimal due to \cite{sabac_optimal}.

The extensions of these results to a general coupled system of nonlocal conservation laws, for the convergent finite volume schemes proposed in \cite{aggarwal2015nonlocal} and also for nonlocal conservation laws with discontinuous flux, are not straightforward and are works in progress. Furthermore, using the Kuznetsov-type lemma  proved in this article, the rate at which the solutions of the nonlocal FTL (see \cite{nftl19,coclite2023nonlocal}) models converge to its continuum limit, can be explored, which we aim to address in our upcoming article.

\section*{Acknowledgement}
The work was carried out during the GV's tenure of the ERCIM ‘Alain Bensoussan’ Fellowship Programme at NTNU.  The project was supported in part by the project \textit{IMod --- Partial differential equations, statistics and data:
An interdisciplinary approach to data-based modelling}, project number 325114, from the Research Council of Norway, and by AA's faculty development allowance, funded by IIM Indore.
\bibliographystyle{abbrv}
\bibliography{AHV}
 \end{document}